\documentclass[12pt]{article}

\usepackage{fullpage}
\usepackage{amssymb}
\usepackage{amsmath}
\usepackage{amsthm}
\usepackage{stmaryrd}
\usepackage{xypic}
\usepackage{setspace}
\usepackage{array}
\usepackage{rotating}
\usepackage{soul}
\usepackage{epstopdf}

\newtheorem{theorem}{Theorem}[section]
\newtheorem{lemma}[theorem]{Lemma}
\newtheorem{proposition}[theorem]{Proposition}
\newtheorem{corollary}[theorem]{Corollary}

\newcommand{\tref}[1]{Theorem~\textup{\ref{thm:#1}}}
\newcommand{\pref}[1]{Proposition~\textup{\ref{prop:#1}}}
\newcommand{\cref}[1]{Corollary~\textup{\ref{cor:#1}}}
\newcommand{\eref}[1]{Equation~\textup{\ref{eq:#1}}}
\newcommand{\lref}[1]{Lemma~\textup{\ref{lem:#1}}}

\newcommand{\Z}{\mathbb Z}

\newcommand{\lcm}{\textrm{lcm}}

\newcommand{\s}{\sigma}
\newcommand{\os}{\hat{\sigma}}
\newcommand{\orho}{\overline{\rho}}

\newcommand{\G}{\Gamma}
\newcommand{\mix}{\diamond}
\newcommand{\comix}{\boxempty}
\newcommand{\comment}[1]{}
\newcommand{\eps}{1}

\newcommand{\D}{\Delta}
\newcommand{\Core}{\operatorname{Core}}

\newcommand{\calP}{\mathcal P}
\newcommand{\calQ}{\mathcal Q}

\begin{document}

\title{Classification of tight regular polyhedra}

\author{Gabe Cunningham \\
University of Massachusetts Boston\\
Boston, Massachusetts 02125,  USA \\
and \\
Daniel Pellicer\\
National University of Mexico\\
Morelia, Michoacan, Mexico
}

\date{ \today }
\maketitle

\begin{abstract}
A regular polyhedron of type $\{p, q\}$ has at least $2pq$ flags, and it is called \emph{tight}
if it has exactly $2pq$ flags. The values of $p$ and $q$ for which there exist tight orientably regular polyhedra were previously known.
We determine for which values of $p$ and $q$ there is a tight non-orientably regular polyhedron of type $\{p, q\}$.
Furthermore, we completely classify tight regular polyhedra in terms of their automorphism groups.



\end{abstract}

\section{Introduction}\label{s:introd}
\emph{Abstract polyhedra} are combinatorial objects that generalize the face-lattice of convex polyhedra. Those possessing the highest degree of symmetry are called \emph{regular polyhedra}.
The face-lattices of platonic solids, known since antiquity, are all regular in this sense, and there
are infinitely many more regular abstract polyhedra.

In \cite{smallest-regular}, Marston Conder
introduced the idea of a \emph{tight} regular polyhedron: any regular polyhedron
with $q$-valent vertices and $p$-gons as faces has at least $2pq$ automorphisms, and the polyhedron is called tight if it has
precisely this number of automorphisms. Tight polyhedra were also studied by the first author
in \cite{tight-polytopes}. In \cite{tight2}, Conder and the first author completely characterized
the values of $p$ and $q$ of tight orientably regular polyhedra, and further generalized to
higher-dimensional analogues.

In the present paper we characterize the degree $q$ of the vertices and the number $p$ of edges in a face of tight non-orientably regular polyhedra. We also take the work from \cite{tight2} one
step further in the following direction. For many values of $p$ and $q$ there are multiple non-isomorphic tight
orientably regular polyhedra with $p$-gonal faces and $q$-valent vertices; here we determine the number of such polyhedra
and describe their automorphism groups.

One of our main results is the following:

\begin{theorem}\label{t:schltypes}
There is a tight regular polyhedron of type $\{p, q\}$ (that is, with $p$-gonal faces and $q$-valent vertices) if and only if one of the following is true:
\begin{enumerate}
\item $p$ and $q$ are both even.
\item $p$ is odd and $q$ is an even divisor of $2p$.
\item $q$ is odd and $p$ is an even divisor of $2q$.
\item $p = 4$ and $q$ is an odd multiple of $3$.
\item $q = 4$ and $p$ is an odd multiple of $3$.
\end{enumerate}
In the second and third cases, there is one such polyhedron up to isomorphism, and it is orientably regular.
In the fourth and fifth cases, there is one such polyhedron up to isomorphism, and it is non-orientably regular.
\end{theorem}

In Sections \ref{s:background} and \ref{s:groups} we review basic concepts and results on tight abstract regular polytopes and their automorphism groups. The classification of orientably regular and non-orientably regular polyhedra are obtained in Sections \ref{s:orientable} and \ref{s:nonorientable}, respectively. Theorem \ref{t:schltypes} follows directly from the results in these two sections.

\section{Background}\label{s:background}
Our definitions are mostly taken from \cite[Chs. 2, 4]{arp}, with some minor modifications.

\subsection{Definition of an abstract polyhedron}

	Let $\calP$ be a ranked partially-ordered set with elements of rank $0$, called \emph{vertices},
	elements of rank $1$, called \emph{edges}, and elements of rank $2$, called \emph{faces}.
	Let us say that two elements $F$ and $G$ are \emph{incident} if $F \leq G$ or $G \leq F$.
	By a \emph{flag} we will mean a maximal chain (totally ordered set). The {\em vertex-figure} at a vertex $F_0$ is $\{G \mid G > F_0\}$. Then, $\calP$ is an \emph{(abstract) polyhedron} if it
	satisfies the following properties:

	\begin{enumerate}
	\item[(1)] Every flag of $\calP$ consists of a vertex, an edge, and a face (all mutually incident).
	\item[(2)] Each edge is incident on exactly two vertices and two faces.
	\item[(3)] The graph determined by the vertex and edge sets is connected.
	\item[(4)] The vertex-figure at every vertex is isomorphic to the vertex and edge lattice of a connected 2-regular graph.
	\end{enumerate}

When considering finite polyhedra, the last property can be interpreted as vertex-figures being (finite) polygons, whereas the second and fourth properties imply that the faces are also polygons.
As a consequence of the second and fourth properties above, given any flag $\Phi$ and $i \in \{0, 1, 2\}$,
there is a unique flag $\Phi^i$ that differs from $\Phi$ only in its element of rank $i$. We say that $\Phi^i$ is
\emph{$i$-adjacent} to $\Phi$ (or simply \emph{adjacent} to $\Phi$ if the rank $i$ is unimportant).

	In the remainder of the paper, let us drop the qualifier ``abstract'' and simply refer to
	polyhedra.

	Given a face of a polyhedron, if it is incident to $p$ edges, then it must also be
	incident to $p$ vertices. These edges and vertices occur in a single cycle, and we say
	that the face is a \emph{$p$-gon}. Similarly, if a vertex is incident to $q$ edges, then it
	is also incident to $q$ faces, occurring in a single cycle. In this case we say that
	the vertex-figure is a $q$-gon. If $\calP$ is a polyhedron whose
	faces are all $p$-gons and whose vertex-figures are all $q$-gons, then
	we say that $\calP$ has \emph{Schl\"afli symbol $\{p, q\}$}, or that it is of \emph{type $\{p, q\}$}.
	A polyhedron that has a Schl\"afli symbol is said to be \emph{equivelar}.
	
	If $\calP$ is a polyhedron, then the \emph{dual} of $\calP$, denoted $\calP^{\delta}$,
	is the polyhedron we obtain by reversing the partial order. If $\calP$ is of type
	$\{p, q\}$, then $\calP^{\delta}$ is of type $\{q, p\}$.

	Given any convex polyhedron, the partially-ordered set of its vertices, edges, and faces, ordered
	by the usual geometric incidence, is an abstract polyhedron. Similarly, any face-to-face tessellation
	of the plane yields an (infinite) abstract polyhedron. Indeed, every abstract polyhedron with finite faces and vertex-figures corresponds
	to a face-to-face tiling of some surface, which may or may not be orientable. The tiling, also called a {\em map}, can be constructed by taking a topological $p$-gon (topological disk with its boundary divided in $p$ segments) for each face $F$ containing $p$ edges. The $p$ segments of the $p$-gon are labeled with the edges incident to $F$, in such a way that if two segments intersect in a point, the corresponding edges in the partial order have a vertex in common. The point of intersection is labeled by the common vertex. Since every edge belongs to two faces, it only remains to identify segments of the $p$-gons corresponding to the same edge in such a way that vertices with the same label are also identified.

On the other hand, some tilings fail to satisfy property (4) above, and therefore they do not correspond to abstract polyhedra. An example of this is the map $\{4,4\}_{(1,1)}$ on the torus (see \cite{coxeter-moser}).

\subsection{Regularity and orientability}
		
	If $\calP$ and $\calQ$ are polyhedra, then a \emph{homomorphism}
	from $\calP$ to $\calQ$ is a function that preserves incidence.
	We say that $\calP$ \emph{covers} $\calQ$ if there is a surjective homomorphism $\varphi$
	from $\calP$ to $\calQ$ that also preserves rank and has the property that
	if flags $\Phi$ and $\Psi$ are $i$-adjacent, then so are their images under $\varphi$.
	An \emph{isomorphism} from $\calP$ to $\calQ$ is an incidence- and rank-preserving bijection.
	An isomorphism from $\calP$ to itself is an \emph{automorphism} of $\calP$, and the group of
	all automorphisms of $\calP$ is denoted by $\G(\calP)$. There is a natural action of
	$\G(\calP)$ on the flags of $\calP$, and due to the connectivity of $\calP$,
	the action of each automorphism is completely determined by its action on any
	given flag.
	
	We say that $\calP$ is \emph{regular} if the natural action of $\G(\calP)$ on the flags
	of $\calP$ is transitive (and hence regular, in the sense of being sharply-transitive).
	Indeed, for convex polyhedra, this definition is equivalent to any of the usual definitions of
	regularity.

	Since each automorphism of $\calP$ is completely determined by its action on any particular
	flag, let us choose a \emph{base flag} $\Phi$ of $\calP$. Then the automorphism group
	$\G(\calP)$ is generated by the \emph{abstract reflections} $\rho_0, \rho_1, \rho_2$,
	where each $\rho_i$ maps $\Phi$ to $\Phi^i$.
	These generators satisfy (at least) the relations $\rho_i^2 = \eps$ for all $i$ and $(\rho_0 \rho_2)^2 = \eps$.
	A regular polyhedron must be equivelar, and if its type is $\{p, q\}$, then $\langle \rho_0, \rho_1 \rangle$ is
	dihedral of order $2p$, and $\langle \rho_1, \rho_2 \rangle$ is dihedral of order $2q$.
	In other words, if $\calP$ is a regular polyhedron of type
	$\{p, q\}$, then $\G(\calP)$ is a smooth quotient of the string Coxeter group $[p, q]$, with presentation
	\[ [p, q] := \langle x, y, z \mid x^2 = y^2 = z^2 = 1, (xy)^p = (yz)^q = (xz)^2 = 1 \rangle. \]

Let $\G = \langle \rho_0, \rho_1, \rho_2 \rangle$ be
	a group such that the generators $\rho_i$ each have order $2$ and such that $(\rho_0 \rho_2)^2 = \eps$.
	Then we say that $\G$ is a \emph{string group generated by involutions} of rank 3, which we will abbreviate
	to \emph{sggi}. Now, for any $I \subseteq \{0, 1, 2\}$, we define $\G_I = \langle \rho_i \mid i \in I \rangle$.
	We say that $\G$ is a \emph{string C-group} of rank 3 if it satisfies the following \emph{intersection condition}:
	\begin{equation}
	\label{eq:int-cond}
	\G_I \cap \G_J = \G_{I \cap J}
	\;\; \textrm{ for all } I,J \subseteq \{0,1, 2\}.
	\end{equation}

If $\Gamma$ is the automorphism group of a polyhedron, the group $\Gamma_I$ corresponds to the stabilizer under the automorphism group of the subset of the base flag consisting of elements with ranks not in $I$. In particular, $\langle \rho_0, \rho_1 \rangle$ is the stabilizer of the base face, and $\langle \rho_1, \rho_2 \rangle$ is the stabilizer of the base vertex. The intersection condition for $\Gamma$ is a consequence of the definition of abstract polyhedron.

	The automorphism group of regular polyhedron is a string C-group of rank 3. Furthermore, there is a natural way to reconstruct a regular polyhedron from its automorphism group and the generators $\rho_i$.
	Indeed, regular polyhedra are in one-to-one correspondence with string C-groups of rank 3. Hence, every string C-group
	is the automorphism group of a (unique) regular polyhedron (see \cite[Thm. 2E11]{arp}).

	We will frequently encounter a group that is clearly an sggi, but where it is unclear whether it is
	a string C-group. The \emph{quotient criterion} below is often useful (\cite[Thm. 2E17]{arp}):

	\begin{proposition}
	\label{prop:quo-crit}
	Let $\G = \langle \rho_0, \rho_1, \rho_{2} \rangle$ be an sggi, and $\Lambda = \langle \lambda_0, \lambda_1,
	\lambda_{2} \rangle$ a string C-group. If there is a homomorphism
	$\pi: \G \to \Lambda$ sending each $\rho_i$ to $\lambda_i$, and if $\pi$ is one-to-one on the subgroup
	$\langle \rho_0, \rho_{1} \rangle$ or the subgroup $\langle \rho_1, \rho_{2} \rangle$,
	then $\G$ is a string C-group.
	\end{proposition}
	
We next state another criterion to determine that some {\em sggi}'s are string C-groups:
	
	\begin{proposition}
	\label{prop:simple-quo}
	Let $\G = \langle \rho_0, \rho_1, \rho_2 \rangle$ be a string C-group. Let $N = \langle (\rho_0 \rho_1)^k \rangle$ or $\langle (\rho_1 \rho_2)^k \rangle$
	for some $k \ge 2$. If $N$ is normal in $\G$, then $\G/N$ is a string C-group.
	\end{proposition}
	
	\begin{proof}
Let $N = \langle (\rho_0 \rho_1)^k \rangle$ and suppose that $N$ is normal.
Let us write $\overline{\rho_i}$ for the image of $\rho_i$ under the canonical projection.
Clearly $\langle \overline{\rho_0}, \overline{\rho_1} \rangle$ and $\langle \overline{\rho_1}, \overline{\rho_2} \rangle$
are both dihedral, and so by \cite[Prop 2E16(a)]{arp}, it suffices to show that
$\langle \overline{\rho_0}, \overline{\rho_1} \rangle \cap \langle \overline{\rho_1}, \overline{\rho_2} \rangle = \langle \overline{\rho_1} \rangle$.
Consider an element in $\langle \overline{\rho_0}, \overline{\rho_1} \rangle \cap \langle \overline{\rho_1}, \overline{\rho_2} \rangle$. We may write that element as $\overline{g}$, where $g \in \G = \langle \rho_0, \rho_1, \rho_2 \rangle$. Then since $\overline{g} \in \langle \overline{\rho_0}, \overline{\rho_1} \rangle$, it follows that $g \in \langle \rho_0, \rho_1 \rangle N$,
and the latter is the same as simply $\langle \rho_0, \rho_1 \rangle$. Similarly, since $\overline{g} \in \langle \overline{\rho_1}, \overline{\rho_2} \rangle$, it follows that $g = h (\rho_0 \rho_1)^{mk}$ for some $h \in \langle \rho_1, \rho_2 \rangle$ and some $m$. But then $g (\rho_0 \rho_1)^{-mk} = h$ is an element of $\langle \rho_0, \rho_1 \rangle$, and so $h$ belongs to the intersection $\langle \rho_0, \rho_1 \rangle \cap \langle \rho_1, \rho_2 \rangle$. Since $\G$ is a string C-group, this means that $h \in \langle \rho_1 \rangle$. Finally, $\overline{g} = \overline{h}$, so we see that $\overline{g} \in \langle \overline{\rho_1} \rangle$. Therefore, $\langle \overline{\rho_0}, \overline{\rho_1} \rangle \cap \langle \overline{\rho_1}, \overline{\rho_2} \rangle$ is contained in $\langle \overline{\rho_1} \rangle$, and the reverse inclusion is obvious.

A dual argument proves the result if $N = \langle (\rho_1 \rho_2)^k \rangle$.
	\end{proof}

	Given a regular polyhedron $\calP$ with automorphism group $\G(\calP) = \langle \rho_0, \rho_1, \rho_2 \rangle$,	
	we define the \emph{abstract rotations} $\s_1 := \rho_0 \rho_1$ and $\s_2 := \rho_1 \rho_2$.
	Then the subgroup $\langle \s_1, \s_2 \rangle$ of $\G(\calP)$ is denoted by $\G^+(\calP)$,
	and called the \emph{rotation subgroup of $\calP$}.
	The index of $\G^+(\calP)$ in $\G(\calP)$ is at most $2$, and when the index is exactly $2$,
	then we say that $\calP$ is \emph{orientably regular}. Otherwise, if $\G^+(\calP) = \G(\calP)$,
	then we say that $\calP$ is \emph{non-orientably regular}.  Indeed, a regular polyhedron is
	orientably or non-orientably regular in accordance with whether the underlying surface
	is orientable or not (when viewing the polyhedron as a map).
	A regular polyhedron $\calP$ is orientably regular if and only if $\G(\calP)$ has a presentation
	in terms of the generators $\rho_0, \rho_1, \rho_2$ such that all of the relators have even length. As
	a consequence, we have the following:
		
	\begin{proposition}
	\label{prop:quos-nonor}
	Let $\calP$ be a non-orientably regular polyhedron. If $\calP$ covers a regular polyhedron $\calQ$, then $\calQ$ is also non-orientably regular.
	\end{proposition}
	
	\begin{proof}
	If $\calP$ is non-orientably regular, then some odd relation holds in $\G(\calP)$, and the same relation must hold in $\G(\calQ)$.
	\end{proof}

	From the properties of the automorphism groups of regular polyhedra and the definitions of $\sigma_1$ and $\sigma_2$ it follows that the rotation subgroups of orientably regular polyhedra satisfy
\begin{equation}\label{eq:sinvolution}
(\sigma_1 \sigma_2)^2 = 1
\end{equation}
and the intersection condition $\langle \sigma_1 \rangle \cap \langle \sigma_2 \rangle = \{1\}$. Indeed, $\langle \sigma_1 \rangle$, $\langle \sigma_1 \sigma_2 \rangle$ and $\langle \sigma_2 \rangle$ are the stabilizers in $\Gamma^+({\cal P})$ of the base face, base edge and base vertex of $\cal P$, respectively.
	Just as we can reconstruct a regular polyhedron from its automorphism group, we can also reconstruct
	an orientably regular polyhedron from its rotation subgroup and specified generators $\sigma_1$ and $\sigma_2$ \cite[Thm. 1]{chiral}.

	Let us say that $\calP$ has \emph{multiple edges} if the underlying graph of $\calP$ has multiple edges with
	the same vertex-set. (In other words, if there is a pair of vertices with more than one edge between them.)
	By regularity, if some pair of vertices has $r$ edges between them, then every pair of vertices has either
	$0$ or $r$ edges between them. Polyhedra without multiple edges are particularly nice to work with
	combinatorially, in part because of the following property.
	
	\begin{proposition}\label{prop:FaithfulAction}
	If $\cal P$ is an orientably regular polyhedron with no multiple edges, then $\G^+(\calP)$ acts faithfully on the vertex set of $\cal P$.
	\end{proposition}
	
	\begin{proof}
	Assume to the contrary that there is a non-trivial automorphism $\gamma$ fixing each vertex. Since $\cal P$ has no multiple edges,
	$\gamma$ must also fix every edge. In particular, $\gamma$ fixes the base edge. Since $\gamma$ fixes the base vertex and the base
	edge, that means that $\gamma \in \langle \s_2 \rangle \cap \langle \s_1 \s_2 \rangle$, and by the intersection
	condition, it follows that $\gamma$ is the identity.
	\end{proof}

Note that for $p \ge 3$, the polyhedron with Schl\"afli type $\{p,2\}$ has no multiple edges and the reflection $\rho_2$ acts like the identity on the vertex set. These are the only polyhedra $\cal P$ with no multiple edges for which the full automorphism group $\Gamma({\cal P})$ does not act faithfully on the vertex set.

	The dual of a regular polyhedron is itself regular. Furthermore, if $\G(\calP) = \langle \rho_0, \rho_1, \rho_2 \rangle$
	and $\G(\calP^{\delta}) = \langle \rho_0', \rho_1', \rho_2' \rangle$, then to obtain the defining relations of $\G(\calP^{\delta})$,
	we can simply change the defining relations of $\G(\calP)$ by replacing each $\rho_i$ with $\rho_{2-i}'$.
	This also has the effect of replacing each $\s_i$ with $(\s_{3-i}^{-1})'$.

\subsection{Tight and flat polyhedra}

	It was shown in \cite[Prop. 3.3]{tight-polytopes} that a finite polyhedron of type $\{p, q\}$ has at least $2pq$ flags.
	When it has exactly that many flags, the polyhedron is called \emph{tight} (a term introduced by Marston Conder
	in \cite{smallest-regular}). Proposition 4.1 of \cite{tight-polytopes} showed that every tight polyhedron
	is also \emph{flat}: every face is incident with every vertex. Furthermore, every flat polyhedron
	has a Schl\"afli symbol and is automatically tight as well.

	A regular polyhedron $\calP$ with $\G(\calP) = \langle \rho_0, \rho_1, \rho_2 \rangle$
	is flat
	if and only if $\G(\calP) = \langle \rho_0, \rho_1 \rangle \langle \rho_1, \rho_2 \rangle$ (\cite[Prop. 4E4]{arp}).
	Equivalently, $\calP$ is flat if and only if $\G(\calP) = \langle \s_1 \rangle \langle \rho_1 \rangle \langle \s_2 \rangle$.
	Let $w \in \G(\calP)$ and suppose that 
	\[ w = \s_1^i \rho_1^j \s_2^k = \s_1^{i'} \rho_1^{j'} \s_2^{k'}. \]
	Then
	\[ \s_1^{i-i'} \rho_1^j = \rho_1^{j'} \s_2^{k'-k}. \]
	Therefore, $w$ is in $\langle \rho_0, \rho_1 \rangle$ and $\langle \rho_1, \rho_2 \rangle$,
	and by the intersection condition (\eref{int-cond}), it follows that $w \in \langle \rho_1 \rangle$.
	So $\s_1^{i-i'} = \eps$ and $\s_2^{k'-k} = \eps$, and thus $\rho_1^j = \rho_1^{j'}$ as well.
	Therefore, the expression of $w$ as $\s_1^i \rho_1^j \s_2^k$ is essentially unique (except that we
	may, of course, change $i$ to $i+p$ and so on.)

	In the remainder of the paper, we will find it useful to use the generating set $\{\s_1, \rho_1, \s_2\}$
	instead of $\{\rho_0, \rho_1, \rho_2\}$.
	Let $\G = \langle \rho_0, \rho_1, \rho_2 \rangle$ be an sggi and let $\sigma_1=\rho_0 \rho_1$, $\sigma_2=\rho_1 \rho_2$. In analogy with regular polyhedra,
	let us say that the group $\G$ is \emph{tight} if $\G = \langle \s_1 \rangle \langle \rho_1 \rangle \langle \s_2 \rangle$.
	If the order of $\s_1$ is $p$ and the order of $\s_2$ is $q$, then we will say that the group $\G$ is \emph{of type $\{p, q\}$}.
	
	The following results all help us determine when a group (or polyhedron) is tight.
	
	\begin{proposition}
	\label{prop:tight-quo-means-tight}
	Suppose $\G = \langle \rho_0, \rho_1, \rho_2 \rangle$ is an sggi with $\sigma_1=\rho_0 \rho_1$ and $\sigma_2=\rho_1 \rho_2$ and with normal subgroup $N = \langle \s_1^m \rangle$ or $N = \langle \s_2^m \rangle$. If $\G / N$ is tight,
	then so is $\G$.
	\end{proposition}

	\begin{proof}
	Without loss of generality, assume that $N = \langle \s_1^m \rangle$. Let $g \in \G$, and let $\varphi: \G \to \G/N$ be
	the canonical map. Let $\overline{g} = \varphi(g)$. Then since $\G/N$ is tight, we may write $\overline{g}$ as
	$\overline{\s_1^i \rho_1^j \s_2^k}$ for some choice of $i$, $j$, and $k$. Then $g$ is
	in the coset $N(\s_1^i \rho_1^j \s_2^k)$, and every element there is in $\langle \s_1 \rangle \langle \rho_1 \rangle \langle \s_2 \rangle$.
	\end{proof}

	\begin{proposition}
	\label{prop:move-s1-s2}
	Let $\G = \langle \rho_0, \rho_1, \rho_2 \rangle$ be an sggi with $\sigma_1=\rho_0 \rho_1$ and $\sigma_2=\rho_1 \rho_2$. Then $\G$ is tight if and only if every expression of the form $\s_2^i \s_1^j$
	is equivalent to an expression of the form $\s_1^{i'} \s_2^{j'}$ or of the form $\s_1^{i'} \rho_1 \s_2^{j'}$.
	\end{proposition}
			
	\begin{proof}
	The necessity is obvious. For sufficiency, we note that the assumption says that we may move any power of $\s_1$ left past
	any power of $\s_2$. Since we also have $\rho_1 \s_1^i = \s_1^{-i} \rho_1$, we see that in any expression of a word in
	the generators of $\G$, we may move every $\s_1$ to the left. Similarly, we may move every $\s_2$ to the right
	(since $\s_2^i \rho_1 = \rho_1 \s_2^{-i}$), and so any element of $\G$ can be written as $\s_1^i \rho_1^j \s_2^k$
	for some $i$, $j$, and $k$.
	\end{proof}

	\begin{proposition}
	\label{prop:tight-covers}
	If $\calP$ and $\calQ$ are polyhedra of type $\{p, q\}$ such that $\calP$ covers $\calQ$, and if $\calP$ is tight, then
	$\calP \simeq \calQ$.
	\end{proposition}
	
	\begin{proof}
	Since $\calP$ is tight, it has $2pq$ flags, and thus $\calQ$ has at most $2pq$ flags. On the other hand, $\calQ$ itself
	has Schl\"afli symbol $\{p, q\}$, and so it has at least $2pq$ flags. The result then follows.
	\end{proof}

\section{Automorphism groups of tight regular polyhedra}\label{s:groups}
Our goal is to find a complete classification of the tight regular polyhedra. In particular,
we want to find, for each Schl\"afli symbol $\{p, q\}$, how many tight regular polyhedra there are of
that type (up to isomorphism), and provide presentations for their automorphism groups.
We will proceed by showing that certain relations must hold, and then that these
relations suffice to define the group.

We will frequently use the following simple result:

\begin{proposition}
\label{prop:reverse-rel}
Let $\G = \langle \rho_0, \rho_1, \rho_2 \rangle$ be an sggi with $\sigma_1=\rho_0 \rho_1$ and $\sigma_2=\rho_1 \rho_2$. Suppose that $g_1 \cdots g_m = h_1 \cdots h_n$,
where each $g_i$ and $h_i$ is in the set $\{\s_1, \rho_1, \s_2\}$. Then
$g_m \cdots g_1 = h_n \cdots h_1$.
\end{proposition}

\begin{proof}
We note that conjugation by $\rho_1$ inverts $\s_1$ and $\s_2$, and it fixes $\rho_1$ (which is the same as
inverting $\rho_1$, since it is an involution). Therefore, conjugating the relation
$g_1 \cdots g_m = h_1 \cdots h_n$ by $\rho_1$, we obtain $g_1^{-1} \cdots g_m^{-1} = h_1^{-1} \cdots h_n^{-1}$.
Inverting both sides then gives the desired result.
\end{proof}

If $\calP$ is a tight regular polyhedron, then every element of $\G(\calP)$ can be written uniquely in the form
$\s_1^i \s_2^j$ or $\s_1^i \rho_1 \s_2^j$, with $i \in \Z_p$ and $j \in \Z_q$. In particular,
$\s_2^{-1} \s_1$ can be written this way. We make the following observation:

\begin{proposition}
\label{prop:i-and-j-subgps}
Let $\G = \langle \rho_0, \rho_1, \rho_2 \rangle$ be an sggi with $\sigma_1=\rho_0 \rho_1$ and $\sigma_2=\rho_1 \rho_2$.
\begin{enumerate}
\item If $\s_2^{-1} \s_1 = \s_1^i \s_2^j$, then $\s_1^{i+1}$ and $\s_2^{j-1}$ are each inverted when conjugating by
$\rho_0$, $\rho_1$, and $\rho_2$. In particular, $\langle \s_1^{i+1} \rangle$ and $\langle \s_2^{j-1} \rangle$
are normal subgroups of $\G$.
\item If $\s_2^{-1} \s_1 = \s_1^i \rho_1 \s_2^j$, then $\s_1^{i-2}$ is inverted when conjugating by $\rho_0$, $\rho_1$ and $\sigma_2$,
and commutes with $\rho_2$, whereas $\s_2^{j+2}$ is inverted when conjugating by $\rho_1$ and $\rho_2$ and $\sigma_1$, and commutes
with $\rho_0$. In particular, $\langle \s_1^{i-2} \rangle$ and $\langle \s_2^{j+2} \rangle$ are
normal subgroups of $\G$.
\end{enumerate}
\end{proposition}

\begin{proof}
By \pref{reverse-rel}, if $\s_2^{-1} \s_1 = \s_1^i \s_2^j$, then also $\s_1 \s_2^{-1} =
\s_2^j \s_1^i$.
Therefore,
\begin{align*}
\s_2^{-1} \s_1^{i+1} &= \s_1^i \s_2^j \s_1^i \\
&= \s_1^{i+1} \s_2^{-1}.
\end{align*}
Thus, $\s_2^{-1}$ commutes with $\s_1^{i+1}$. Since $\s_2^{-1} = \rho_2 \rho_1$ and
conjugation by $\rho_1$ inverts $\s_1^{i+1}$, it follows that conjugation by $\rho_2$ also inverts $\s_1^{i+1}$, and $\langle \sigma_1^{i+1} \rangle$ is normal. A similar idea using $\sigma_2^j \sigma_1^i \sigma_2^j$ instead of $\sigma_1^i \sigma_2^j \sigma_1^i$ proves that $\langle \sigma_2^{j-1} \rangle$ is normal.

For the second part we use the elements $\sigma_1^i \rho_1 \sigma_2^j \rho_1 \sigma_1^i$ and $\sigma_2^j \rho_1 \sigma_1^i \rho_1 \sigma_2^j$ to show that $\sigma_2 \sigma_1^{i-1} = \sigma_1^{-i+1} \sigma_2$ and that $\sigma_1 \sigma_2^{-j-1} = \sigma_2^{j+1} \sigma_1^{-1}$, respectively. Using Equation (\ref{eq:sinvolution}) it can now be verified that $\sigma_2 \sigma_1^{i-2} = \sigma_1^{-i+2} \sigma_2$ and $\sigma_1 \sigma_2^{-j-2} = \sigma_2^{j+2} \sigma_1$, and the statement follows.
\end{proof}

\begin{theorem}
\label{thm:first-rels}
Let $\calP$ be a tight regular polyhedron of type $\{p, q\}$. If $\calP$ is orientably regular, then for some
$i$ and $j$, the group $\G(\calP)$ is the quotient of $[p, q]$ by the extra relation $\s_2^{-1} \s_1 = \s_1^i \s_2^j$.
If $\calP$ is non-orientably regular, then for some $i$ and $j$, the relation $\s_2^{-1} \s_1 = \s_1^i \rho_1 \s_2^j$ holds.
\end{theorem}

\begin{proof}
Let $\G$ be the quotient of $[p, q]$ by the relation $\s_2^{-1} \s_1 = \s_1^i \s_2^j$. By \pref{i-and-j-subgps} (a), the subgroups
$\langle \s_1^{i+1} \rangle$ and $\langle \s_2^{j-1} \rangle$ are normal. In the quotient of $\G$ by both of these subgroups,
the relation $\s_2^{-1} \s_1 = \s_1^{-1} \s_2$ holds, and by  \pref{reverse-rel}, the relation $\sigma_1 \sigma_2^{-1} = \sigma_2 \sigma_1^{-1}$
also holds. Furthermore, $\s_2 \s_1 = \s_1^{-1} \s_2^{-1}$ in the automorphism group of any regular polyhedron. 
Therefore
\[ \s_2^r \s_1 = \s_2^{r-1} \s_1^{-1} \s_2^{-1} = \s_2^{r-2} \s_1 \s_2^{-2}, \]
and it follows that 
for any $r$, $\s_2^r \s_1 = \s_1^{(-1)^r} \s_2^{-r}$. Therefore, $\s_2^r \s_1^s = \s_1^{s(-1)^r} \s_2^{r(-1)^s}$,
and by \pref{move-s1-s2}, it follows that this quotient is tight. Then by two applications of \pref{tight-quo-means-tight}, we see that $\G$ is itself tight.
Furthermore, note that the relations of $\G$ are all even.

Now, let $\calP$ be a tight regular polyhedron of type $\{p, q\}$. Then for some $i$ and $j$, either the relation
$\s_2^{-1} \s_1 = \s_1^i \s_2^j$ holds or the relation $\s_2^{-1} \s_1 = \s_1^i \rho_1 \s_2^j$ holds.
If $\calP$ is orientably regular, it must be the former, since the latter relation is odd. The above analysis shows that
this relation alone is enough to guarantee tightness, and so by \pref{tight-covers}, $\G(\calP)$ must be this quotient of $[p, q]$.
On the other hand, if $\calP$ is non-orientably regular, then the relation $\s_2^{-1} \s_1 = \s_1^i \rho_1 \s_2^j$
must hold, since otherwise, $\G(\calP)$ would be the group $\G$ above, all of whose relations are even.
\end{proof}

So we see that for tight orientably regular polyhedra, their automorphism groups are single-relator quotients of
string Coxeter groups. The same is not true, in general, of tight non-orientably regular polyhedra. However,
two extra relations always suffice. We need the following lemmas.

\begin{lemma}
\label{lem:second-rel}
Let $\calP$ be a tight non-orientably regular polyhedron, with $\G(\calP) = \langle \s_1, \rho_1, \s_2 \rangle$
and $\G(\calP^{\delta}) = \langle \os_1, \orho_1, \os_2 \rangle$. Then either $\s_2^{-2} \s_1 = \s_1^a \s_2^b$ holds in $\G(\calP)$ (for some $a$ and $b$), or $\os_2^{-2} \os_1 = \os_1^a \os_2^b$ holds in $\G(\calP^{\delta})$.
\end{lemma}

\begin{proof}
Since $\calP$ is tight, either $\s_2^{-2} \s_1 = \s_1^a \s_2^b$, or else $\s_2^{-2} \s_1 = \s_1^a \rho_1 \s_2^b$.
In the first case, we are done. Otherwise, consider $\calP^{\delta}$. Each relation of $\G(\calP)$ yields a relation
in $\G(\calP^{\delta})$ by sending $\rho_1$ to $\orho_1$ and $\s_k$ to $\os_{3-k}^{-1}$. So if the relation
$\s_2^{-2} \s_1 = \s_1^a \rho_1 \s_2^b$ holds in $\G(\calP)$, it follows that $\os_1^2 \os_2^{-1} = \os_2^{-a} \orho_1 \os_1^{-b}$
holds in $\G(\calP^{\delta})$, and from this it follows (by \pref{reverse-rel}) that
$\os_2^{-1} \os_1^2 = \os_1^{-b} \orho_1 \os_2^{-a}$.  Now, $\calP^{\delta}$ is also a tight non-orientably regular polyhedron, so
\tref{first-rels} implies that the relation $\os_2^{-1} \os_1 = \os_1^i \orho_1 \os_2^j$ holds in $\G(\calP^{\delta})$ for some $i$ and $j$.
Then, working in $\G(\calP^{\delta})$ and using \pref{i-and-j-subgps} (b), we get that:
\begin{align*}
\os_2^{-2} \os_1 &= \os_2^{-1} (\os_2^{-1} \os_1) \\
&= \os_2^{-1} \os_1^i \orho_1 \os_2^j \\
&= \os_2^{-1} \os_1^{i-2} \os_1^2 \orho_1 \os_2^j \\
&= \os_1^{2-i} \os_2^{-1} \os_1^2 \orho_1 \os_2^j \\
&= \os_1^{2-i} \os_1^{-b} \orho_1 \os_2^{-a} \orho_1 \os_2^j \\
&= \os_1^{2-i-b} \os_2^{a+j},
\end{align*}
and so a relation of the desired type holds in $\G(\calP^{\delta})$.

\end{proof}

\begin{lemma}
\label{lem:a-and-b-subgps}
Let $\G = \langle \rho_0, \rho_1, \rho_2 \rangle$ be an sggi with $\sigma_1 = \rho_0 \rho_1$ and $\sigma_2 = \rho_1 \rho_2$.
Suppose that $\G$ satisfies the relations $\s_2^{-1} \s_1 = \s_1^i \rho_1 \s_2^j$
and $\s_2^{-2} \s_1 = \s_1^a \s_2^b$. Then $\s_2^{-1} \s_1^2 = \s_1^{i-a} \rho_1 \s_2^{b-j-2}$.
Furthermore, conjugation by $\s_2$ inverts $\s_1^{i-3-a}$ and $\s_1$ commutes with $\s_2^{b-2}$, and in
particular, the subgroups $\langle \s_1^{i-3-a} \rangle$ and $\langle \s_2^{b-2} \rangle$ are normal.
\end{lemma}

\begin{proof}
By \pref{i-and-j-subgps} (b), conjugation by $\s_1$ inverts $\s_2^{j+2}$. Therefore,
\begin{align*}
\s_2^{-1} \s_1^2 &= \s_1^i \rho_1 \s_2^j \s_1 \\
&= \s_1^i \rho_1 \s_2^{-2} \s_2^{j+2} \s_1 \\
&= \s_1^i \rho_1 \s_2^{-2} \s_1 \s_2^{-(j+2)} \\
&= \s_1^i \rho_1 \s_1^a \s_2^{b-j-2} \\
&= \s_1^{i-a} \rho_1 \s_2^{b-j-2}.
\end{align*}
It follows from \pref{reverse-rel} that $\s_1^2 \s_2^{-1} = \s_2^{b-j-2} \rho_1 \s_1^{i-a}$. Therefore,
\begin{align*}
\s_2^{-1} \s_1^2 \rho_1 \s_1^{i-a} &= \s_1^{i-a} \rho_1 \s_2^{b-j-2} \rho_1 \s_1^{i-a} \\
&= \s_1^{i-a} \rho_1 \s_1^2 \s_2^{-1}.
\end{align*}
Then $\s_2^{-1} \s_1^{2-i+a} = \s_1^{i-a-2} \s_2$.
Therefore, by (\ref{eq:sinvolution}),
\begin{align*}
\s_2 \s_1^{3-i+a} &= \s_1^{-1} \s_2^{-1} \s_1^{2-i+a}
= \s_1^{i-a-3} \s_2.
\end{align*}
It follows that $\langle \s_1^{i-a-3} \rangle$ is normal. Finally, since $\s_2^{-2} \s_1 = \s_1^a \s_2^b$, then also
$\s_1 \s_2^{-2} = \s_2^b \s_1^a$ (by \pref{reverse-rel}).
Therefore,
\begin{align*}
\s_2^{b-2} \s_1 &= \s_2^b \s_1^a \s_2^b \\
&= \s_1 \s_2^{b-2},
\end{align*}
and so $\s_2^{b-2}$ is normalized by $\s_1$.
\end{proof}

\begin{theorem}
\label{thm:second-rel}
Let $\calP$ be a tight non-orientably regular polyhedron of type $\{p, q\}$. Then either $\G(\calP)$ is the quotient of $[p, q]$
by the relations $\s_2^{-1} \s_1 = \s_1^i \rho_1 \s_2^j$ and
$\s_2^{-2} \s_1 = \s_1^a \s_2^b$ (for some choice of $i$, $j$, $a$, and $b$), or $\G(\calP^{\delta})$ is the quotient of $[q, p]$ by those relations.
\end{theorem}

\begin{proof}
Let us define $\Delta(p,q)_{(i,j,a,b)}$ to be the quotient of $[p, q]$ by the relations $\s_2^{-1} \s_1 = \s_1^i \rho_1 \s_2^j$
and $\s_2^{-2} \s_1 = \s_1^a \s_2^b$. \lref{second-rel} implies that either $\G(\calP)$ is a quotient of $\Delta(p,q)_{(i,j,a,b)}$ or that $\G(\calP^{\delta})$ is a quotient of some $\Delta(q,p)_{(i,j,a,b)}$,
for some choice of $(i,j,a,b)$.
Without loss of generality, let us assume that the first is true. It remains to show that $\G(\calP)$ is equal to $\Delta(p,q)_{(i,j,a,b)}$,
and not to a proper quotient. For that, it suffices to show that $\Delta(p,q)_{(i,j,a,b)}$ is itself tight.
In light of \pref{tight-quo-means-tight}, we may take the quotient by any normal subgroup generated by a power
of $\s_1$ or $\s_2$, and if that quotient is tight, then so is $\Delta(p,q)_{(i,j,a,b)}$. There are several such
normal subgroups; in particular, \pref{i-and-j-subgps} (b) shows that $\s_1^{i-2}$ and $\s_2^{j+2}$ generate
normal subgroups, and \lref{a-and-b-subgps} shows that $\s_1^{i-3-a}$ and $\s_2^{b-2}$ generate normal subgroups.
Taking the quotient by these subgroups yields the group $\Delta(p', q')_{(2,-2,-1,2)}$ for some $p'$
dividing $p$ and some $q'$ dividing $q$. 	
Now, in this quotient, $\s_2^{-2} \s_1 = \s_1^{-1} \s_2^2$, and therefore $\s_2^{-3} \s_1 = \s_2^{-1} \s_1^{-1} \s_2^2 = \s_1 \s_2^3$.
It follows that $\langle \s_2^3 \rangle$ is normal. Similarly, the relation $\s_2^{-1} \s_1 = \s_1^2 \rho_1 \s_2^{-2}$ holds, and so
\begin{align*}
\s_2^{-1} \s_1^4 &= \s_1^2 \rho_1 \s_2^{-2} \s_1^3 \\
&= \s_1^2 \rho_1 \s_1^{-1} \s_2^2 \s_1^2 \\
&= \s_1^2 \rho_1 \s_1^{-1} \s_2 \s_1^{-1} \s_2^{-1} \s_1 \\
&= \s_1^2 \rho_1 \s_1^{-1} \s_1^{-2} \rho_1 \s_2^2 \s_2^{-1} \s_1 \\
&= \s_1^5 \s_2 \s_1 \\
&= \s_1^4 \s_2^{-1},
\end{align*}
and thus $\langle \s_1^4 \rangle$ is normal as well. Taking the quotient by these subgroups yields $\D(4,3)_{(2,-2,-1,2)}$.
Using GAP \cite{gap}, we can verify that this group is tight; in fact, it is the group of the hemicube. It follows that $\D(p,q)_{(i,j,a,b)}$ is tight,
proving the claim.
\end{proof}

We see that every tight regular polyhedron has as its automorphism group one of the groups in
Theorem \ref{thm:first-rels} or \ref{thm:second-rel}. Furthermore, the given groups are always
tight, in the sense that $\G = \langle \s_1 \rangle \langle \rho_1 \rangle \langle \s_2 \rangle$.
Two things remain to be determined for each family of groups. First, for which values of the
parameters is the group a string C-group? Second, under what conditions is the group
actually of type $\{p, q\}$; in other words, when is there no collapse of the subgroups
$\langle \rho_0, \rho_1 \rangle$ and $\langle \rho_1, \rho_2 \rangle$? The answer to
these questions is quite dissimilar in the orientable case versus the non-orientable case,
and we require fairly different methods for the two cases. 

\section{Tight orientably regular polyhedra}\label{s:orientable}
We first consider the classification of tight, orientably regular polyhedra. Part of the classification was completed in \cite[Thm. 3.4]{tight2}:

	\begin{theorem}
	\label{thm:tight-polyhedra}
	There is a tight orientably regular polyhedron of type $\{p, q\}$ if and only if
	one of the following is true$\hskip 1pt :$
	\begin{enumerate}
	\item $p$ and $q$ are both even, or \\[-15 pt]
	\item $p$ is odd and $q$ is an even divisor of $2p$, or \\[-15 pt]
	\item $q$ is odd and $p$ is an even divisor of $2q$.
	\end{enumerate}
	\end{theorem}

Furthermore, it was proved in \cite[Thm. 3.3]{tight2} that if $p$ or $q$ is odd, then there is at most one
isomorphism type of tight orientably regular polyhedra of type $\{p, q\}$. What remains to be determined
is how many tight orientably regular polyhedra there are when $p$ and $q$ are both even, and to
find presentations for their automorphism groups.

Let $\Lambda(p,q)_{i,j}$ be the quotient of $[p, q]$ by the extra relation $\s_2^{-1} \s_1 = \s_1^i \s_2^j$.
We determined in \tref{first-rels} that if $\calP$ is a tight orientably regular polyhedron of type $\{p, q\}$, then it
has automorphism group $\Lambda(p,q)_{i,j}$ for some choice of $i$ and $j$. For a given Schl\"afli symbol $\{p, q\}$,
we need to determine which values of $i$ and $j$ make $\Lambda(p,q)_{i,j}$ the automorphism group of a tight
orientably regular polyhedron of type $\{p, q\}$.

Let us recall that if $H$ is a subgroup of a group $G$, the largest subgroup of $H$ which is normal in $G$ is
called {\em the core} of $H$ in $G$, and we shall denote it $\Core_{G}(H)$.
If $\Core_G(H)$ is trivial, then we say that $H$ is \emph{core-free} in $G$.

We are able to reduce our classification problem using the following result.

\begin{proposition}
\label{prop:tight-mix}
$\Lambda(p,q)_{i,j}$ is the automorphism group of a tight orientably regular polyhedron
of type $\{p, q\}$ if and only if there are values $p'$ and $q'$ such that
\begin{enumerate}
\item $p'$ divides $p$ and $i+1$,
\item $q'$ divides $q$ and $j-1$,
\item the group $\Lambda(p,q')_{i,1}$ is the automorphism group of a tight orientably regular polyhedron
of type $\{p, q'\}$, with $\langle \s_2 \rangle$ core-free in $\Lambda(p,q')_{i,1}$, and
\item the group $\Lambda(p',q)_{-1,j}$ is the automorphism group of a tight orientably regular polyhedron
of type $\{p', q\}$, with $\langle \s_1 \rangle$ core-free in $\Lambda(p',q)_{-1,j}$.
\end{enumerate}
\end{proposition}

\begin{proof}
First, suppose that $\Lambda(p,q)_{i,j}$ is the automorphism group of a tight orientably regular polyhedron
of type $\{p, q\}$. Let $N := \langle \s_1^{p'} \rangle$ be the core of $\langle \s_1 \rangle$ in $\Lambda(p, q)_{i,j}$,
with $p'$ dividing $p$. \pref{i-and-j-subgps} (a) shows that $\langle \s_1^{i+1} \rangle$ is normal,
and it follows that $p'$ divides $i+1$ as well. Then the quotient
of $\Lambda(p,q)_{i,j}$ by $N$ is $\Lambda(p',q)_{-1,j}$, and the latter has $\langle \s_1 \rangle$ core-free. 
Since $\s_2$ has order $q$ in $\Lambda(p,q)_{i,j}$ (by supposition), the same is true in $\Lambda(p',q)_{-1,j}$;
otherwise, we would have some $\s_2^{q'} \in \langle \s_1^{p'} \rangle$, violating the intersection condition.
Then \pref{simple-quo} implies that $\Lambda(p',q)_{-1,j}$
is a string C-group. Therefore, $\Lambda(p',q)_{-1,j}$ is the automorphism group of a tight orientably regular
polyhedron of type $\{p', q\}$. An analogous argument, taking the quotient by the core of $\langle \s_2 \rangle$ 
(which we set to be $\langle \s_2^{q'} \rangle$), shows that $\Lambda(p, q')_{i,1}$ is the automorphism group of a 
tight orientably regular polyhedron of type $\{p, q'\}$, where $q'$ divides $q$ and $j-1$.

In the other direction, suppose that $\Lambda(p,q')_{i,1}$ and $\Lambda(p',q)_{-1,j}$ are automorphism groups
of tight orientably regular polyhedra of types $\{p, q'\}$ and $\{p', q\}$, respectively, and suppose that
$p'$ divides $p$ and $i+1$, and that $q'$ divides $q$ and $j-1$. It is clear from the presentations that
$\Lambda(p,q)_{i,j}$ covers $\Lambda(p',q)_{i,j}$ and $\Lambda(p,q')_{i,j}$. Since $p'$ divides
$i+1$, it follows that $i \equiv -1$ (mod $p'$), and thus $\Lambda(p', q)_{i,j} = \Lambda(p', q)_{-1,j}$.
Similarly, $\Lambda(p,q')_{i,j} = \Lambda(p,q')_{i,1}$. So $\Lambda(p,q)_{i,j}$ covers
$\Lambda(p,q')_{i,1}$ and $\Lambda(p',q)_{-1,j}$. Since $\Lambda(p,q')_{i,1}$ has type $\{p, q'\}$,
it follows that $\s_1$ has order $p$ in $\Lambda(p,q)_{i,j}$, and since $\Lambda(p',q)_{-1,j}$ has type
$\{p', q\}$, it follows that $\s_2$ has order $q$ in $\Lambda(p,q)_{i,j}$. Finally,
since the cover from $\Lambda(p,q)_{i,j}$ to $\Lambda(p',q)_{-1,j}$ is one-to-one
on the facets and the latter is a string C-group, \pref{quo-crit} implies that $\Lambda(p,q)_{i,j}$
is also a string C-group. So $\Lambda(p,q)_{i,j}$ is the automorphism group of an orientably
regular polyhedron $\calP$ of type $\{p, q\}$. Since the group $\Lambda(p,q)_{i,j}$ is
tight (by \tref{first-rels}), it follows that $\calP$ is tight.
\end{proof}

There is a nice combinatorial interpretation of what it means for $\langle \s_2 \rangle$
to be core-free in $\G(\calP)$. We start by remarking that $\langle \s_2 \rangle$ is 
core-free in $\G(\calP)$ if and only if it is core-free in $\G^+(\calP)$, since
$\rho_1$ and $\rho_2$ normalize any subgroup $\langle \s_2^k \rangle$,
and so if $\s_1 (= \rho_0 \rho_1)$ normalizes such a subgroup, then so does $\rho_0$.
Therefore, we can work with $\G^+(\calP)$ instead. We start with a few simple results.

\begin{proposition}
\label{prop:fix-all-nbrs}
Let $\calP$ be a tight polyhedron.
Suppose $\varphi \in \G(\calP)$ acts as a rotation at the vertex $v$; that is, $\varphi$ fixes $v$ while cyclically permuting
the neighbors of $v$. If $\varphi$ fixes some neighbor of $v$, then it fixes all neighbors of $v$.
\end{proposition}

\begin{proof}
Suppose that $\varphi$ fixes $u$, and let
$w$ be another neighbor of $v$. Then there is some automorphism $\psi$ that acts as a rotation
at $v$ and with the property that $u \psi = w$. Furthermore, since $\varphi$ and $\psi$ both
act as rotations at $v$, it follows that $u \varphi \psi = u \psi \varphi$. Therefore,
\[ w \varphi = u \psi \varphi = u \varphi \psi = u \psi = w, \]
and so $\varphi$ fixes every neighbor of $v$.
\end{proof}

\begin{corollary}
\label{cor:fix-all-verts}
Let $\calP$ be a tight orientably regular polyhedron. Let $\varphi \in \G^+(\calP)$,
and suppose that $\varphi$ fixes some vertex $v$ and one of the neighbors of $v$.
Then $\varphi$ fixes all vertices of $\calP$.
\end{corollary}

\begin{proof}
Let $\varphi \in \G^+(\calP)$ and suppose that $\varphi$ fixes some vertex $v$ and one of its neighbors.
Then by \pref{fix-all-nbrs}, $\varphi$ fixes all of the neighbors
of $v$. The stabilizer of each vertex is a dihedral group, and since $\calP$ is orientable and
$\varphi \in \G^+(\calP)$, it follows that $\varphi$ acts at a rotation at each of those neighbors. 
Therefore, $\varphi$ fixes the neighbors
of each neighbor of $v$. Continuing in this manner and using the connectivity of $\calP$,
we see that $\varphi$ fixes every vertex.
\end{proof}

\begin{proposition}
\label{prop:s2-core}
Let $\calP$ be a tight orientably regular polyhedron of type $\{p, q\}$, with $\G^+(\calP) = \langle \s_1, \s_2 \rangle$.
Let $v$ be the base vertex of $\calP$, and let $q'$ be the smallest positive integer such that $\s_2^{q'}$
fixes one of the neighbors of $v$. Then $\langle \s_2^{q'} \rangle$ is the subgroup of $\G^+(\calP)$ that fixes every vertex 
of $\calP$, and $\langle \sigma_2^{q'} \rangle = \Core_{\G^+({\cal P})}(\langle \sigma_2 \rangle)$.
\end{proposition}

\begin{proof}
By \cref{fix-all-verts}, $\s_2^{q'}$ must fix every vertex of $\calP$, from which it is immediate that
$\langle \sigma_2^{q'} \rangle \lhd \G^+(\calP)$. Furthermore, if $\varphi \in \G^+(\calP)$ fixes every
vertex of $\calP$, then it must lie in $\langle \s_2 \rangle$, and by our choice of $q'$ it follows that
$\langle \s_2^{q'} \rangle$ is the subgroup of $\G^+(\calP)$ that fixes every vertex.

Next we want to show that $\langle \s_2^{q'} \rangle$ is the largest subgroup of $\langle \s_2 \rangle$
that is normal in $\G^+(\calP)$. Suppose that
$\sigma_2^a \notin \langle \sigma_2^{q'} \rangle$. Then there is a vertex $u$ such that $u\sigma_2^a \ne u$.
Since $\calP$ is tight, the base face is incident on every vertex, and thus we can find some
$b \in \mathbb{Z}$ such that $v\sigma_1^b = u$. Then $v \sigma_1^b\sigma_2^a \sigma_1^{-b} =
u\sigma_2^a\sigma_1^{-b} \ne v$. Therefore, $\s_1^b \s_2^a \s_1^{-b} \notin \langle \s_2 \rangle$ (since $\s_2$ fixes $v$),
and so $\langle \sigma_2^a \rangle$ is not normal in $\G^+({\cal P)}$.
\end{proof}

We are now ready to explain the connection between having multiple edges and the
core of $\langle \s_2 \rangle$. 

\begin{proposition}
\label{prop:mult-edge-core}
Let $\calP$ be a tight orientably regular polyhedron. Then $\calP$ has no multiple edges if and
only if $\langle \s_2 \rangle$ is core-free in $\G(\calP)$.
\end{proposition}

\begin{proof}
First, suppose that $\langle \s_2 \rangle$ has a nontrivial core. Then by \pref{s2-core}, there
is an automorphism $\s_2^{q'}$ that fixes all vertices, and by \pref{FaithfulAction},
it follows that $\calP$ has multiple edges.

Conversely, suppose that $\calP$ has multiple edges, and let $v$ be the base vertex. 
Let $u$ be a neighbor of $v$, and let $e_1$ and $e_2$ be edges between $u$ and $v$.
By the regularity of $\calP$, there is some nontrivial even automorphism $\varphi$ that sends the pair 
$(v, e_1)$ to $(v, e_2)$, and since $\varphi$ fixes $v$, it follows that $\varphi = \s_2^{q'}$ for some $q'$.
In order for $\s_2^{q'}$ to send $e_1$ to $e_2$, it must be the case that $\varphi$ fixes $u$. 
Then by \cref{fix-all-verts}, $\s_2^{q'}$ fixes every vertex, which implies that $\langle \s_2^{q'} \rangle$
is normal in $\langle \s_2 \rangle$. Thus $\langle \s_2 \rangle$ has a nontrivial core.
\end{proof}

In light of \pref{mult-edge-core}, we may reinterpret \pref{tight-mix} as follows:

\begin{corollary}\label{c:NoMultEd}
$\Lambda(p,q)_{i,j}$ is the automorphism group of a tight orientably regular polyhedron
of type $\{p, q\}$ if and only if there are values $p'$ and $q'$ such that
\begin{enumerate}
\item $p'$ divides $p$ and $i+1$,
\item $q'$ divides $q$ and $j-1$,
\item $\Lambda(p,q')_{i,1}$ is the automorphism group of a tight orientably regular polyhedron
of type $\{p, q'\}$, and with no multiple edges, and
\item $\Lambda(p',q)_{-1,j}$ is the automorphism group of a tight orientably regular polyhedron
of type $\{p', q\}$, such that the dual has no multiple edges.
\end{enumerate}
\end{corollary}

Our first step will then be to find all tight orientably-regular polyhedra of type $\{p, q\}$ with no multiple edges.

\subsection{Tight orientably regular polyhedra with no multiple edges}

	There is exactly one orientably regular polyhedra $\calP$ with Schl\"afli type $\{p,2\}$ for every $p \ge 3$. This can be easily seen either by showing that the graph induced by the vertex and edge set must be connected and $2$-regular, or by noting that the Coxeter group $[p,2]$ is isomorphic to $D_p \times C_2$ and must cover $\Gamma(\calP)$. None of these polyhedra has multiple edges.
	
	In what follows we shall
	determine the remaining tight orientably regular polyhedra with no multiple edges.
	
	In the results that follow, we will generally assume the following, which we call the \emph{usual setup} (see Figure \ref{fig:UsualSetup}).
	Let $\cal Q$ be a tight orientably regular polyhedron of type $\{p, q\}$, with $q \ge 3$, and suppose that
	$\calQ$ has no multiple edges.
	Let us fix a base face $F_1$ and label the vertices with elements of $\Z_p$ in such a way that
	$i \sigma_1 = i + 1$. The flag $\Phi$ will consist of the vertex $1$, the edge between $0$ and $1$, and the
	face $F_1$. Let $F_2$ be the other face containing the edge between $0$ and $1$, and let $k$ be the other
	vertex of $F_2$ that is adjacent to $1$ (so that $0 \sigma_2 = k$).
	
\begin{figure}
\begin{center}
\includegraphics[width=5cm, height=3.1cm]{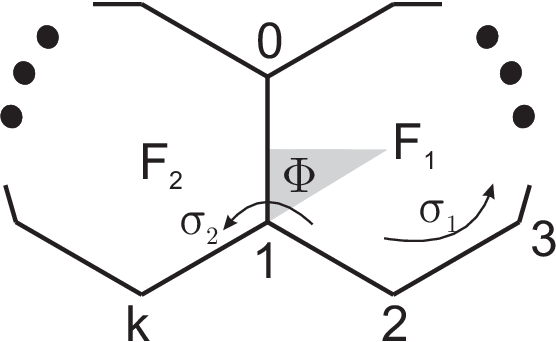}
\end{center}
\caption{The \emph{usual setup} for tight orientably regular polyhedra with no multiple edges \label{fig:UsualSetup}}
\end{figure}
	
	\begin{lemma}\label{l:p-k-restrictions}
	Let $\cal Q$ be a tight orientably regular polyhedron of type $\{p, q\}$ with $q \geq 3$, no multiple edges, and with the 
	usual setup. Then
	\begin{itemize}
	\item[(a)] $p > q$.
	\item[(b)] $p$ is even.
	\item[(c)] Every vertex of the dual of $\mathcal{Q}$ has exactly two neighbors.
	\item[(d)] $k$ is even.
	\item[(e)] The vertices of $F_2$, in clockwise order, are $(1, 0, -k+1, -k, -2k+1, -2k, \ldots, k+1, k)$.
	\end{itemize}
	\end{lemma}
	
	\begin{proof}
	A tight polyhedron of type $\{p, q\}$ has $p$ vertices, and each vertex has $q$ neighbors; so in order to have no multiple edges, it
	must be that $p > q$, proving part (a).

	Recall that the base flag $\Phi$ consists of vertex $1$, the edge between $0$ and $1$, and $F_1$. The involutory automorphism
	$\gamma = \rho_0 \s_1^k$ maps $\Phi$ to the flag $\Psi$ consisting of vertex $k$, the edge between
	$k$ and $k+1$ and the face $F_1$. Note that $\gamma$ fixes $F_1$ and maps respectively the vertices $0$ and $1$ to $k+1$ and $k$. Then
	$F_2 (= F_1\sigma_2)$ is mapped to a face $F'$ sharing the edge between $k$ and $k+1$ with $F_1$.
	Furthermore, $\gamma$ fixes the edge between $1$ and $k$ (since it swaps their endpoints and there are no
	multiple edges), and since $\gamma$ sends $F_2$ to $F'$, it follows that $F'$ also contains that edge.
	Let $\Upsilon$ be the flag containing vertex $1$, the edge between $1$ and $k$, and
	the face $F_2$. Then $\gamma$ maps $\Upsilon$ to the flag containing vertex $k$, the edge between $1$ and $k$, and $F'$.
	If $F' \neq F_2$, then $\gamma$ maps $\Upsilon$ to $\Upsilon^{0,2}$. But $\cal Q$ is orientable, and $\gamma$ is an odd
	automorphism. So it must be that $F' = F_2$ so that $\gamma$ maps $\Upsilon$ to $\Upsilon^0$.
	Hence $F_2$ contains the edges between $0$ and $1$, between $1$ and $k$, and between $k$ and $k+1$.
	
	Now, the automorphism $\sigma_1^k$ fixes $F_1$, and it maps the edge between $0$ and $1$ to the edge between $k$ and $k+1$.
	Since $F_1$ and $F_2$ share both of those edges, it follows that $\s_1^k$ also fixes $F_2$.
	Therefore, $F_2$ also contains the edge between $k+1$ and $2k$, since that is the image of the edge between $1$
	and $k$. Finally, an inductive procedure shows that $F_2$ contains the edge between $nk$ and $nk+1$ and the edge between $nk+1$
	and $(n+1)k$ for every $n$. In particular, $F_2$ shares every other edge with $F_1$.
	
	If $p$ were odd, then $F_2$ would have to share every edge with $F_1$. Then there could only be two faces, which would
	imply that $q = 2$. Since $q \ge 3$, the parameter $p$ must be even. Furthermore, this means that $F_2$ shares half of its
	edges with $F_1$, and half of its edges with some other face. By regularity, every face must share its edges with exactly two
	distinct faces, which means that in the dual of $\calQ$, every vertex has exactly two neighbors.

	Just as $F_2$ shares half of its edges with $F_1$, the face $F_1$ shares half of its edges with $F_2$. If $F_1$
	shared two consecutive edges with $F_2$, then it would have to share all of them (by regularity), and so it must
	share every other edge with $F_2$. Since the two faces share the edge between $0$ and $1$, it follows that, for every $i$, they
	share the edge between $2i$ and $2i+1$ but not the edge from $2i$ to $2i-1$. Since they also share the edge between $k$ and $k+1$,
	it follows that $k$ is even, proving part (d). Part (e) immediately follows.
	\end{proof}

	\begin{lemma}\label{l:theface}
	Let $\cal Q$ be a tight orientably regular polyhedron of type $\{p, q\}$, with the usual setup.
	Then:
	\begin{itemize}
	  \item[(a)] \[i\sigma_2 = \left\{\begin{array}{ll}
	\frac{k(2-i)}{2}& \mbox{if $i$ is even}\\
	1+\frac{k(1-i)}{2}& \mbox{if $i$ is odd}.
	\end{array}\right.\]
	   \item[(b)] $(k/2)^2 \equiv 1$ modulo $p/2$.
	\end{itemize}
	\end{lemma}	

	\begin{proof}
	The automorphism $\s_2$ sends vertex $0$ to vertex $k$ and fixes $1$. Proceeding clockwise around $F_1$ and applying $\s_2$ gives us the vertices
	of $F_2$ in clockwise order. From Lemma \ref{l:p-k-restrictions}, the clockwise order of the vertices in $F_2$ is $(1,0,-k+1,-k,-2k+1,-2k,\dots,k+1,k)$, and part (a) follows.

	For part (b), note first that since Lemma \ref{l:p-k-restrictions}
	says that every vertex of the dual of $\cal Q$ has exactly two neighbors, it follows that $\Core_{\G^+({\cal Q})}(\langle \sigma_1 \rangle)
	= \langle \sigma_1^2 \rangle$. In particular, this means that $\s_2^{-1} \s_1^2 \s_2 = \s_1^{2s}$ for some $s$.
	Now, since $(\s_2 \s_1)^2 = 1$,
	\begin{align*}
	\s_1^2 &= (\s_2 \s_1)^{-2} \s_1^2 (\s_2 \s_1)^2 \\
	&= (\s_1^{-1} \s_2^{-1} \s_1^{-1}) \s_1^{2s} (\s_1 \s_2 \s_1) \\
	&= \s_1^{-1} \s_2^{-1} \s_1^{2s} \s_2 \s_1 \\
	&= \s_1^{-1} \s_1^{2s^2} \s_1 \\
	&= \s_1^{2s^2}.
	\end{align*}
	Therefore $2s^2 \equiv 2$ (mod $p$); that is, $s^2 \equiv 1$ (mod $p/2$).
	Now, $F_1\sigma_2 = F_2$ and therefore, with the labeling of vertices as in
	Lemma \ref{l:theface}, $2\sigma_2 = 0$, $1\sigma_2 =1$, $0\sigma_2 = k$, $(p-1)\sigma_2 = k+1$ and so on. Then
	\[0 = 2\sigma_2 = 0\sigma_1^2 \sigma_2 = 0\sigma_2 \sigma_1^{2s} = k\sigma_1^{2s},\]
	and so $s = -k/2$. This implies that $1 \equiv (-k/2)^2 \equiv (k/2)^2$ (mod $p/2$).
	\end{proof}
			
	Lemma \ref{l:theface} establishes the order of the vertices in $F_2$, which is the face sharing the edge between $0$ and $1$ with $F_1$ , and it determines
	the action of $\sigma_2$ on the vertices. It follows from \pref{FaithfulAction} that, in order to characterize all tight orientably regular polyhedra with no
	multiple edges, we only need to determine all possible values of $p$ and $k$. (Note that the value of $q$ plays no role on the expressions
	of $\sigma_1$ and $\sigma_2$.)
	
	Lemma \ref{l:theface} (b) imposes a strong condition on the value of $k$. The following lemma suggests how restrictive this condition is. The proof is straightforward and omitted, but see \cite[Section 1.2]{number-of-roots} for the number $a(n)$ of solutions of $x^2=1$ in $\mathbb{Z}_n^*$.
	
	\begin{lemma}\label{l:primes}
	Let $P$ be a prime and $n$ a positive integer, and let $X_{P,n}$ be the set of integers $1 \le x \le P^{n-1}$ satisfying that $x^2 \equiv 1$ modulo $P^n$. Then
	\begin{itemize}
	  \item[(a)] $X_{2,1} = \{1\}$, $X_{2,2} = \{1,3\}$ and $X_{2,n} = \{1,2^{n-1}-1,2^{n-1}+1,2^n-1\}$ if $n\ge 3$.
	  \item[(b)] $X_{P,n} = \{1,P^n-1\}$ if $P$ is odd.
	\end{itemize}
	\end{lemma}
	
	In general, if $p = P_1^{\alpha_1} \cdots P_s^{\alpha_s}$ with $P_1, \dots, P_s$ distinct primes and $P_1=2$, then $(k/2)^2 \equiv 1$ (mod $p/2$) if and only if $(k/2)^2 \equiv 1$ (mod $2^{\alpha_1-1})$ and $(k/2)^2 \equiv 1$ (mod $P_i^{\alpha_i}$) for $i\ge 2$.
	
	We now obtain $q$ from $p$ and $k$. The value of $q$ is the order of $\sigma_2$, or alternatively, the smallest positive $m$ such that $2\sigma_2^m = 2$. An inductive procedure shows that for $m \ge 2$,
	\[2\sigma_2^{m} = 2\left(k/2 - \left(k/2\right)^2 + \left(k/2\right)^3 - \cdots + \left(-1\right)^{m}\left(k/2\right)^{m-1}\right).\]
	This implies that $q$ is the smallest positive $m$ satisfying that
	\[2 \equiv 2\left(k/2 - \left(k/2\right)^2 + \left(k/2\right)^3 - \cdots + \left(-1\right)^{m-1}\left(k/2\right)^{m}\right) \quad \mbox{(mod $p$)},\]
	which by Lemma \ref{l:theface} (b) is equivalent to
	\begin{equation}\label{eq:Orderm}
	2 \sigma_2^m \equiv \left\{\begin{array}{ll}
	                  \left(k-2\right)(m-1)/2 & \mbox{if $m$ is odd}\\
	                  k+\left(k-2\right)(m-2)/2 & \mbox{if $m$ is even}
	                  \end{array} \right. \quad \mbox{(mod $p$).}
	\end{equation}
	Therefore $q$ is the smallest positive $m$ satisfying (\ref{eq:Orderm}).
	
	Now we are ready to state our main results about tight orientably regular polyhedra with no multiple edges.
	
	\begin{proposition}\label{p:qOdd}
	Let $\mathcal{Q}$ be a tight orientably regular polyhedron with no multiple edges and Schl\"afli type $\{p,q\}$ with $q$ odd. Then $p=2q$. Furthermore, it is unique up to isomorphism.
	\end{proposition}
	
	\begin{proof}
	From Equation (\ref{eq:Orderm}) we know that if $q$ is odd then $2\equiv(k/2-1)(q-1)$ (mod $p$). Multiplying both sides by $(k/2+1)$ yields
	that $k+2 \equiv ((k/2)^2-1)(q-1)$ (mod $p$), and since $(k/2)^2 \equiv 1$ (mod $p/2$) and $q$ is odd, it follows that $k+2 \equiv 0$ (mod $p$); in other
	words, the only choice for $k$ is $k = p-2$. Substituting in Equation (\ref{eq:Orderm}) we obtain that $-2(q-1) \equiv 2$ (mod $p$) and hence $p$ divides $2q$.
	Since $p$ is even and $p > q$ (by Lemma \ref{l:p-k-restrictions}), it follows that $p = 2q$. Such a polyhedron is unique up to isomorphism since $k$ is determined by the value of $p$
	and by the fact that $q$ is odd.
	\end{proof}

This result is also a consequence of \cite[Thm.3.4]{tight2}, since if $p$ is a proper divisor of $2q$ then $p \le q$ and $\cal Q$ must have multiple edges.

	\begin{proposition}\label{p:4items}
	Let $\mathcal{Q}$ be a tight orientably regular polyhedron with no multiple edges and Schl\"afli type $\{p,q\}$ with $q \ge 4$ even. Let $p=P_1^{\alpha_1} P_2^{\alpha_2} \cdots P_s^{\alpha_s}$ with $P_i$ prime, $P_i < P_{i+1}$ and $\alpha_i \ge 1$ for all $i$. Then
	\begin{itemize}
	  \item[(a)] $P_1 = 2$;
	  \item[(b)] the maximal power of $2$ that divides $q$ is either $2$, $4$ or $2^{\alpha_1-1}$, and it is $4$ only if $\alpha_1 \ge 3$;
	  \item[(c)] for every $i \ge 2$, either $q$ is coprime with $P_i$ or $P_i^{\alpha_i}$ divides $q$; and
	  \item[(d)] $q$ divides $p$.
	\end{itemize}
	\end{proposition}
	
	\begin{proof}
	Part (a) follows from Lemma \ref{l:p-k-restrictions} (b).
	
	From Equation (\ref{eq:Orderm}) and Lemma \ref{l:theface} (b) we have that $(k-2)q/2 \equiv 0$ (mod $p$). This implies that $(k/2-1)q/2\equiv 0$ modulo $p/2$. This is equivalent to
	\begin{equation}\label{eq:qkprimes}
	(k/2-1)q/2\equiv 0 \,\,\mbox{(mod $2^{\alpha_1-1}$)} \quad \mbox{and} \quad (k/2-1)q/2 \equiv 0 \,\,\mbox{(mod $P_i^{\alpha_i}$) for all $i \ge 2$.}
	\end{equation}
	
	If $\alpha_1\ge 4$, Lemma \ref{l:primes} (a) implies that $k/2-1$ is congruent to either $0$, $2^{\alpha_1-2}-2$, $2^{\alpha_1-2}$ or $-2$ (mod $2^{\alpha_1-1}$). First note that if $k/2-1 \equiv 0$ (mod $2^{\alpha_1-1}$) then any even value of $q$ satisfies the left hand side of Equation (\ref{eq:qkprimes}). On the other hand, if $k/2-1 \equiv 2^{\alpha_1-2}-2$, (resp. $k/2-1 \equiv 2^{\alpha_1-2}$, $k/2-1 \equiv -2$) then $2^{\alpha_1-2}$ (resp. $2$, $2^{\alpha_1-2}$) divides $q/2$. Conversely, if $2^{\alpha_1-2}$ (resp. $2$, $2^{\alpha_1-2}$) divides $q/2$ then the left part of Equation (\ref{eq:qkprimes}) is satisfied whenever $k/2-1 \equiv 2^{\alpha_1-2}-2$, (resp. $k/2-1 \equiv 2^{\alpha_1-2}$, $k/2-1 \equiv -2$).

If $\alpha_1 = 3$ then $k/2-1$ is congruent to either $0$ or $2$ modulo $4$. In the case when $k/2-1 \equiv 0$ (mod $4$) no restriction is imposed to $q/2$, whereas if $k/2-1 \equiv 2$ (mod $4$) then $q/2$ must be even. If $\alpha_1 \in \{1,2\}$ then (\ref{eq:qkprimes}) is always satisfied and no restriction is imposed on $q/2$.
	
	From Lemma \ref{l:primes} (b) we know that $k/2 \equiv \pm1$ modulo $P_i^{\alpha_i}$ for all $i\ge 2$. Therefore $k/2-1$ is congruent to either $0$ or $-2$ modulo $P_i^{\alpha_i}$. Since $2$ is coprime with $P_i^{\alpha_i}$ we observe that if $k/2-1 \equiv -2$ modulo $P_i^{\alpha_i}$, then in order to satisfy the right hand side of Equation (\ref{eq:qkprimes}), $P_i^{\alpha_i}$ must divide $q/2$, and any even value of $q$ with this property will work. Otherwise, if $P_i^{\alpha_i}$ divides $k/2-1$ then any even $q$ satisfies the right hand side of Equation (\ref{eq:qkprimes}).
	
	Since $q/2$ is the smallest positive integer satisfying that $(k/2-1)q/2 \equiv 0$ modulo $p/2$, the only factors of $q/2$ are those required by the restrictions in the previous three paragraphs. In particular, if $k/2 \equiv 1$ (mod $2^{\alpha_1-1}$) (resp. to $2^{\alpha_1-2}-1$, $2^{\alpha_1-2}+1$ or $-1$) then $q/2$ is odd (resp. $2^{\alpha_1-1}$, $4$ or $2^{\alpha_1-1}$ is the maximal power of $2$ dividing $q$), implying (b). Furthermore, if $i \ge 2$ and $k/2 \equiv 1$ (mod $P_i^{\alpha_i}$) then $P_i$ does not divide $q$, implying (c). Hence all factors of $q$ are also factors of $p$ and (d) holds.
	\end{proof}
	
	We are now ready to fully characterize the tight orientably regular polyhedra with no multiple edges.
	
	\begin{theorem}\label{th:RegularTight}
	Let $p = P_1^{\alpha_1} P_2^{\alpha_2} \cdots P_s^{\alpha_s}$ with $P_1 = 2$, $P_1, \dots, P_s$ distinct primes and each $\alpha_i$ a positive integer. For any even $q$ with $4 \le q < p$ satisfying (b), (c) and (d) of Proposition \ref{p:4items} there exists a tight orientably regular polyhedron with no multiple edges and type $\{p,q\}$. The polyhedron is unique unless $\alpha_1 \ge 4$ and $2^{\alpha_1-1}$ divides $q$, in which case there are two such polyhedra. Moreover, every tight regular polyhedron with no multiple edges either has one of these types, or has type $\{2q,q\}$ for some odd $q$, or it corresponds to the map of type $\{p,2\}$ on the sphere.
	\end{theorem}
	
	\begin{proof}
	We already know by Proposition \ref{p:qOdd} that if $q$ is odd then $p = 2q$ and $k=-2$.
	
	If $q \ge 4$ is even then we find $k/2 \in \mathbb{Z}_{p/2}$ as a solution of the congruences
	\begin{eqnarray*}
	k/2 \equiv 1 \mbox{\, (mod $2^{\alpha_1-1}$)} && \mbox{if $q/2$ is odd ($\alpha_1 \ge 2$)},\\
	k/2 \equiv 2^{\alpha_1-2}-1,-1 \mbox{\, (mod $2^{\alpha_1-1}$)} && \mbox{if $2^{\alpha_1-2}$ divides $q/2$ ($\alpha_1 \ge 4$)},\\
	k/2 \equiv 2^{\alpha_1-2}+1 \mbox{\, (mod $2^{\alpha_1-1}$)} && \mbox{if $q/2$ is even but not divisible by $4$ ($\alpha_1 \ge 2$),} \\
	k/2 \equiv 1 \mbox{\, (mod $P_i^{\alpha_i}$)} && \mbox{if $P_i$ is odd and it does not divide $q$},\\
	k/2 \equiv -1 \mbox{\, (mod $P_i^{\alpha_i}$)} && \mbox{if $P_i$ is odd and it divides $q$}.
	\end{eqnarray*}
	
	This gives a unique solution (mod $p/2$) unless $\alpha_1 \ge 4$ and $2^{\alpha_1-1}$ divides $q$, where there are two solutions. Multiplying by $2$ we obtain $k$.
	
	It remains to be shown that there exists a tight regular polyhedron for all such parameters $p$ and $k$. Having chosen $p$, $q$, and $k$, Lemma \ref{l:p-k-restrictions}
	describes the order of the vertices around $F_2$, and Lemma \ref{l:theface} describes the action of $\s_2$ on the vertices (and in particular, it describes the neighbors
	of vertex $1$). We need to show that these choices actually yield a polyhedron. Arguing analogously to Lemma \ref{l:p-k-restrictions}, it can be shown that
	if $x$ and $y$ are two consecutive neighbors of $1$, then the order of the vertices in the face determined by these adjacencies is $(1,y,y-x+1,2y-x,2y-2x+1,3y-2x,\dots,x-y+1,x)$.
	In other words, half of the edges go from a vertex $i$ to $i+y-1$, and half go from a vertex $j$ to $j-x+1$. With the $q$ faces defined that way it is easy to verify that every edge
	belongs to precisely two such faces, as a consequence of the fact that if $x$ is a neighbor of $1$ then so is $2-x$ (by applying the automorphism $\rho_1$).
	It also follows that the order of the faces around neighboring vertices is the same, just reversing the orientation.
	This shows that these $q$ faces suffice and that the diamond condition and strong flag connectivity hold.
	\end{proof}

	We conclude by describing the automorphism groups of the polyhedra we have found.
	We determined earlier that the automorphism group of a tight orientably regular polyhedron with no multiple
	edges is $\Lambda(p,q)_{i,1}$ for some choice of $i$. Labeling the vertices as usual (using some parameter $k$), we have that
	\[1 = 0\sigma_1 = k\sigma_2^{-1} \sigma_1 = k\sigma_1^i \sigma_2^j,\]
	which implies that $k \sigma_1^i = 1\sigma_2^{-j} = 1$. Since also $k \s_1^i = k+i$, it follows that $i = -k+1$.
	
\subsection{Full classification}

	We now return to the discussion of determining all tight orientably regular polyhedra of type $\{p, q\}$. Corollary \ref{c:NoMultEd} implies that all such polyhedra cover tight orientably regular polyhedra with types $\{p,q'\}$ and $\{p',q\}$, with the property that the former and the dual of the latter have no multiple edges. On the other hand, there is only one tight orientably regular polyhedron of type $\{p,q\}$ having such quotients.

\begin{proposition}\label{p:Unicity}
Let $\cal P$ and $\cal Q$ be tight orientably regular polyhedra of type $\{p,q\}$ such that both cover a polyhedron of type $\{p,q'\}$ with no multiple edges, and a polyhedron of type $\{p',q\}$ whose dual has no multiple edges, for some $q'$ dividing $q$ and some $p'$ dividing $p$. Then $\cal P$ and $\cal Q$ are isomorphic.
\end{proposition}

\begin{proof}
We know that $\Gamma({\cal P}) = \Lambda(p,q)_{i,j}$ and $\Gamma({\cal Q}) = \Lambda(p,q)_{i',j'}$ for some $i, i', j, j'$. We need to show that $i=i'$ and $j=j'$.

Let ${\cal K}_1$ and ${\cal K}_2$ be the polyhedra with types $\{p,q'\}$ and $\{p',q\}$, respectively. Then $\Gamma({\cal K}_1) = \Lambda(p,q')_{i_1,1}$ and $\Gamma({\cal K}_2) = \Lambda(p',q)_{-1,j_2}$. Clearly ${\cal K}_1$ and ${\cal K}_2$ are quotients of $\cal P$ (and of $\cal Q$) by $\langle \sigma_2^{k_1}\rangle$ and by $\langle \sigma_1^{k_2} \rangle$, respectively, for some $k_1$ and $k_2$. The relations $\sigma_2^{-1} \sigma_1 = \sigma_1^i \sigma_2^j$ and $\sigma_2^{-1} \sigma_1 = \sigma_1^{i_1} \sigma_2$ of $\Lambda(p,q)_{i,j}$ and $\Lambda(p,q')_{i_1,1}$ imply that $i_1 = i$; similarly, $j_2 = j$. But the same is true for $i'$ and $j'$ and so $\Gamma({\cal P}) \equiv \Gamma({\cal Q}) \equiv \Lambda(p,q)_{i_1,j_2}$.
\end{proof}

Proposition \ref{p:Unicity} implies that, in order to determine all tight orientably regular polyhedra with type $\{p,q\}$ we only need to determine all quotient maps satisfying the requirements of Corollary \ref{c:NoMultEd}.

First, for each $q'$ dividing $q$, we determine the values of $i$ such that $\Lambda(p,q')_{i,1}$ is the automorphism group of a tight orientably regular polyhedron of type $\{p, q'\}$ with no multiple edges. Then, for each $p'$ dividing $p$, we determine the values of $j$ such that $\Lambda(p', q)_{-1,j}$ is the automorphism group of a tight orientably regular polyhedron whose dual has no multiple edges. Finally, we determine which pairs of these polyhedra satisfy the conditions of \pref{tight-mix} that $p'$ divides $i+1$ and $q'$ divides $j-1$.

	Let us illustrate the procedure by determining all tight orientably regular polyhedra of type $\{48, 32\}$.
	Proposition \ref{p:4items} implies that all possible values of $q'$ (that yield a tight orientably regular polyhedron of type $\{48, q'\}$ with no multiple edges)
	are $2$, $4$ or $8$. Solving the congruences in the proof of Theorem \ref{th:RegularTight} we find that $k = -i+1$ is $2$ when $q' = 2$, is $26$ when $q'=4$,
	and is $14$ or $38$ if $q'=8$. Since $k = -i + 1$, this gives us that the values of $i$ are $-1$, $23$, $35$ and $11$, respectively.
	
	To find the tight orientably regular polyhedra of type $\{p', 32\}$ such that $\langle \s_1 \rangle$ is core-free, we will work with the dual polyhedron of type $\{32, p'\}$.
	Then Proposition \ref{p:4items} implies that all possible values of $p'$ are $2$, $4$ or $16$. Solving the congruences in the proof of Theorem \ref{th:RegularTight}
	we find that $k$ is $2$ when $p' = 2$, is $18$ when $p' = 4$, and is $14$ or $30$ if $p' = 16$. Therefore the values of $i$ are $-1$, $15$, $19$ and $3$, respectively.
	To return to the dual polyhedron of type $\{p', 32\}$, we note that the relation $\s_2^{-1} \s_1 = \s_1^i \s_2^j$ in $\G(\calP)$ yields the relation
	$\s_1 \s_2^{-1} = \s_2^{-i} \s_1^{-j}$ in $\G(\calP^{\delta})$, and by \pref{reverse-rel}, this is equivalent to the relation $\s_2^{-1} \s_1 =
	\s_1^{-j} \s_2^{-i}$. So when considering the dual polyhedron with type $\{p',32\}$, we must substitute $i$ by $-j$. This gives that the values of $j$ are $1$, $17$, $13$ and $29$, respectively.

	Examining all of the possible values for $p'$, $i$, $q'$, and $j$ yields the following 10 groups whose parameters
	satisfy the conditions of \pref{tight-mix}:
	\begin{align*}
	\Lambda(48, 32)_{-1,1} && \Lambda(48, 32)_{-1,13} \\
	\Lambda(48, 32)_{-1,17} && \Lambda(48, 32)_{-1,29} \\
	\Lambda(48, 32)_{11,1} && \Lambda(48, 32)_{11,17} \\
	\Lambda(48, 32)_{23,1} && \Lambda(48, 32)_{23,17} \\
	\Lambda(48, 32)_{35,1} && \Lambda(48, 32)_{35,17} \\
	\end{align*}	

Note that it is always possible to pick $i = -1$ and $j = 1$; indeed, this is the group named $\G(p, q)$ in \cite{tight-polytopes}, and it is the group of the polyhedron $\{p, q \mid 2\}$ (see \cite[p. 196]{arp}).
It is clear that if there are many tight orientably regular polytopes of type $\{p,q\}$ then the factorizations of $p$ and $q$ in primes have several factors in common. On the other hand, whenever $p$ and $q$ are relatively
prime, the only tight orientably regular polytope of type $\{2p,2q\}$ is $\{2p, 2q \mid 2\}$ with group $\Lambda(2p, 2q)_{-1, 1}$. 

\section{Tight non-orientably regular polyhedra}\label{s:nonorientable}
We now consider the classification of tight, non-orientably regular polyhedra.

	In \tref{second-rel}, we saw that every tight non-orientably regular polyhedron $\calP$ of type $\{p, q\}$ has
	automorphism group $\G(\calP) = \Delta(p,q)_{(i,j,a,b)}$ or its dual, where $\Delta(p,q)_{(i,j,a,b)}$ is the quotient
	of $[p, q]$ by the extra relations $\s_2^{-1} \s_1 = \s_1^i \rho_1 \s_2^j$ and $\s_2^{-2} \s_1 = \s_1^a \s_2^b$.
	It remains to determine which such groups actually appear as the automorphism group of a tight
	non-orientably regular polyhedron. First, we note the following:
				
	\begin{proposition}
	\label{prop:nonor-no-2}
	Let $\calP$ be a non-orientably regular polyhedron of type $\{p, q\}$, with automorphism group $\G(\calP) = \langle \s_1, \rho_1, \s_2 \rangle$.
	Then neither $\langle \s_1^2 \rangle$ nor $\langle \s_2^2 \rangle$ is normal.
	\end{proposition}

	\begin{proof}
	Without loss of generality, let $N = \langle \s_1^2 \rangle$ and suppose that $N$ is normal in $\G(\calP)$.
	Then by \pref{simple-quo}, $\G(\calP) / N$ is a string C-group, and therefore it the automorphism group of a polyhedron $\calQ$
	of type $\{2, q\}$. \pref{quos-nonor} says that since $\calP$ is non-orientably regular, so is $\calQ$. But there is only
	a single polyhedron of type $\{2, q\}$, and it is orientably regular. Thus, $N$ cannot be normal after all.
	\end{proof}
		
	We now work to find restrictions on the parameters $(i, j, a, b)$. We start with several technical lemmas.

	\begin{lemma}
	\label{lem:b-equiv-2}
	Let $\calP$ be a tight non-orientably regular polyhedron of type $\{p, q\}$ with $\G(\calP) = \D(p,q)_{(i,j,a,b)}$.
	If $q$ is odd, then $b = 2$, and if $q$ is even, then $b = 2$ or $2 + q/2$.
	\end{lemma}
	
	\begin{proof}
	\lref{a-and-b-subgps} established that $\s_1 \s_2^{b-2} = \s_2^{b-2} \s_1$. Therefore,
	\[ \s_2^{b-2} (\s_2^{-1} \s_1) = (\s_2^{-1} \s_1) \s_2^{b-2}. \]
	On the other hand,
	\begin{align*}
	\s_2^{b-2} (\s_2^{-1} \s_1) &= \s_2^{b-2} (\s_1^i \rho_1 \s_2^j) \\
	&= (\s_1^i \rho_1 \s_2^j) \s_2^{2-b} \\
	&= (\s_2^{-1} \s_1) \s_2^{2-b}.
	\end{align*}
	It follows that $\s_2^{b-2} = \s_2^{2-b}$, and thus $b-2 \equiv 2-b$ (mod $q$).
	Therefore, $2b \equiv 4$ (mod $q$), and the result then follows.
	\end{proof}

	\begin{lemma}
	\label{lem:a-and-j-restrictions}
	Let $\calP$ be a tight non-orientably regular polyhedron of type $\{p, q\}$ with $\G(\calP) = \D(p,q)_{(i,j,a,b)}$. Then
	\begin{enumerate}
	\item $a = 1 + p/2$
	\item If $p = 4$, then $j = 1$ or $j = 1 + q/2$. If $p \neq 4$, then $j = 1 + q/2$.
	\end{enumerate}
	\end{lemma}

	\begin{proof}
	In $\D(p,q)_{(i,j,a,b)}$, the relations $\s_2^{-1} \s_1 = \s_1^i \rho_1 \s_2^j$ and $\s_2^{-2} \s_1 = \s_1^a \s_2^b$ both hold.
	Furthermore, \lref{a-and-b-subgps} says that $\s_2^{-1} \s_1^2 = \s_1^{i-a} \rho_1 \s_2^{b-j-2}$. Using these relations and the fact (from
	\pref{i-and-j-subgps} (b)) that conjugation by $\s_2$ inverts $\s_1^{i-2}$, we get that
	\begin{align*}
	\s_1^a \s_2^b &= \s_2^{-2} \s_1 \\
	&= \s_2^{-1} \s_1^i \rho_1 \s_2^j \\
	&= \s_2^{-1} \s_1^{i-2} \s_1^2 \rho_1 \s_2^j \\
	&= \s_1^{2-i} \s_2^{-1} \s_1^2 \rho_1 \s_2^j \\
	&= \s_1^{2-i} \s_1^{i-a} \rho_1 \s_2^{b-j-2} \rho_1 \s_2^j \\
	&= \s_1^{2-a} \s_2^{2j-b+2}.
	\end{align*}	
	Thus we see that $\s_1^{2a-2} = \s_2^{2j-2b+2}$. Since $\calP$ is a polyhedron, $\langle \s_1 \rangle \cap \langle \s_2 \rangle = \{ \eps \}$,
	and it follows that $2a - 2 \equiv 0$ (mod $p$) and that $(2j-2b+2) \equiv 0$ (mod $q$).

	Now, without loss of generality, $0 \leq a \leq p-1$. So, since $2a \equiv 2$ (mod $p$), it follows that either
	$a = 1$ or $a = 1 + p/2$. If $a = 1$, then we have that
	\[ \s_2^{-2} \s_1 = \s_1 \s_2^b, \]
and $\s_2^b$ is a conjugate of $\s_2^{-2}$.
Now, if $q$ is odd, then $b = 2$ (by \lref{b-equiv-2}), and it follows that $\langle \s_2^2 \rangle$ is normal.
On the other hand, if $q$ is even, then $\s_2^2$ has order $q/2$, and so does $\s_2^b$.
This implies that $b$ is even and again $\langle \s_2^2 \rangle$ is normal. But by \pref{nonor-no-2}, that cannot
	happen. Therefore, it must be that $a = 1 + p/2$.

	Similarly, we have that $2b \equiv 2j+2$ (mod $q$), and by \lref{b-equiv-2}, $2b \equiv 4$ (mod $q$). Therefore,
	$2j+2 \equiv 4$ (mod $q$), and so $2j \equiv 2$ (mod $q$). Thus, either $j = 1$ or $j = 1 + q/2$. Now,
	if $j = 1$, then \pref{i-and-j-subgps} (b) says that $N = \langle \s_2^3 \rangle$ is normal. In the quotient
	of $\G(\calP)$ by $N$, the order of $\s_1$ is still $p$, and we have that
	\begin{align*}
	\s_1^{1+p/2} \s_2^b &= \s_2^{-2} \s_1 \\
	&= \s_2 \s_1 \\
	&= \s_1^{-1} \s_2^{-1}.
	\end{align*}
	So $\s_1^{2+p/2} = \s_2^{-b-1}$. Since $\G(\calP)$ is a string C-group, \pref{simple-quo} implies that
	$\G(\calP) / N$ is a string C-group as well, and so again $\langle \s_1 \rangle \cap \langle \s_2 \rangle = \{ \eps \}$.
	So $2 + p/2 \equiv 0$ (mod $p$), from which it follows that $p = 4$. So if $p \neq 4$,
	then $j = 1 + q/2$.
	\end{proof}

	\begin{lemma}
	\label{lem:s14-s26-normal}
	Let $\calP$ be a tight non-orientably regular polyhedron of type $\{p, q\}$ with $\G(\calP) = \D(p,q)_{(i,j,a,b)}$. Then
	the subgroups $\langle \s_1^4 \rangle$ and $\langle \s_2^6 \rangle$ are normal.
	\end{lemma}

	\begin{proof}
	\pref{i-and-j-subgps} (b) says that conjugation by $\s_2$ inverts $\s_1^{i-2}$, and \lref{a-and-b-subgps} says that conjugation by $\s_2$
	inverts $\s_1^{i-3-a}$. It follows that $\s_2$ inverts $\s_1^{a+1}$ and therefore, it also inverts $\s_1^{2a+2}$.
	Since $a = 1 + p/2$ (by \lref{a-and-j-restrictions}), it follows that conjugation by $\s_2$ inverts $\s_1^4$, and so
	$\langle \s_1^4 \rangle$ is normal.
	
	For the other claim, \pref{i-and-j-subgps} (b) says that $\langle \s_1^{j+2} \rangle$ is normal, and so $\langle \s_1^{2j+4} \rangle$
	is also normal. By \lref{a-and-j-restrictions}, $j = 1$ or $j = 1 + q/2$. In any case, $2j+4 \equiv 6$ (mod $q$), and so
	$\langle \s_2^6 \rangle$ is normal.
	\end{proof}

	\begin{corollary}
	\label{cor:p-and-q-divisors}
	Let $\calP$ be a tight non-orientably regular polyhedron of type $\{p, q\}$ with $\G(\calP) = \D(p,q)_{(i,j,a,b)}$. Then
	$p$ is divisible by $4$ and $q$ is divisible by $3$.
	\end{corollary}
	
	\begin{proof}
	\lref{s14-s26-normal} says that $\langle \s_1^4 \rangle$ and $\langle \s_2^6 \rangle$ are both normal, and \pref{nonor-no-2} says
	that neither $\langle \s_1^2 \rangle$ and $\langle \s_2^2 \rangle$ is normal. It follows that $\s_1^2 \notin \langle \s_1^4 \rangle$
	and that $\s_2^2 \notin \langle \s_2^6 \rangle$. Therefore, $p$ is a multiple of $4$ and $q$ is a multiple of $3$.
	\end{proof}

	\begin{lemma}
	\label{lem:value-of-i}
	Suppose that $\calP$ is a tight non-orientably regular polyhedron of type $\{p, q\}$ with $\G(\calP) = \D(p,q)_{(i,j,a,b)}$,
	If $p/4 \equiv 3$ (mod $4$), then $i = p/4 - 1$, and otherwise $i = 3p/4 - 1$.
	\end{lemma}
	
	\begin{proof}
	\pref{i-and-j-subgps} (b) says that $\langle \s_1^{i-2} \rangle$ is normal, and \lref{s14-s26-normal} says that
	$\langle \s_1^4 \rangle$ is normal. Because of \pref{nonor-no-2}, no subgroup of $\langle \s_1 \rangle$ containing $\langle \s_1^4 \rangle$ properly
	is normal, and thus $\langle \s_1^{i-2} \rangle$ must be contained in $\langle \s_1^4 \rangle$. It follows that
	$i \equiv 2$ (mod 4).

	Now, suppose we take the quotient of $\G(\calP)$ by $\langle \s_2^{j+2} \rangle$ and $\langle \s_2^{b-2} \rangle$.
	This has the effect of replacing $j$ with $-2$ and $b$ with $2$ without changing $i$, $a$, or
	$p$. Then
	\begin{align*}
	\s_2 \s_1^4 &= \s_1^{-1} \s_2^{-1} \s_1^3 \\
	&= \s_1^{i-1} \rho_1 \s_2^{-2} \s_1^2 \\
	&= \s_1^{i-1} \rho_1 \s_1^a \s_2^2 \s_1 \\
	&= \s_1^{i-1} \rho_1 \s_1^a \s_2 \s_1^{-1} \s_2^{-1} \\
	&= \s_1^{i-1} \rho_1 \s_1^{a-i} \rho_1 \s_2 \\
	&= \s_1^{2i-a-1} \s_2.
	\end{align*}
	Since $a = 1 + p/2$, we have that $\s_2 \s_1^4 \s_2^{-1} = \s_1^{2i-2-p/2}$. On the other hand, \lref{s14-s26-normal} says that $\s_2 \s_1^4 \s_2^{-1} = \s_1^{-4}$,
	so it follows that $2i - 2 - p/2 \equiv -4$ (mod $p$). Therefore, $2i \equiv p/2 - 2$ (mod $p$), and thus $i \equiv p/4 - 1$ (mod $p/2$).
	Thus we see that $i = p/4 - 1$ or $i = 3p/4 - 1$. In order for $i \equiv 2$ (mod $4$), we need to pick $i = p/4 - 1$ if $p/4 \equiv 3$ (mod 4), and
	otherwise we need to pick $i = 3p/4 - 1$.
	\end{proof}	

	\begin{lemma}
	\label{lem:value-of-b}
	Suppose that $\calP$ is a tight non-orientably regular polyhedron of type $\{p, q\}$ with $\G(\calP) = \D(p,q)_{(i,j,a,b)}$,
	Then $b = 2$.
	\end{lemma}

	\begin{proof}
	First, suppose that $p = 4$. Then $a = 3$ by \lref{a-and-j-restrictions}, and $i = 2$ by \lref{value-of-i}. Therefore,
	\begin{align*}
	\s_1^3 \s_2^b &= \s_2^{-2} \s_1 \\
	&= \s_2^{-1} \s_1^2 \rho_1 \s_2^j \\
	&= \s_2^{-1} \s_1^{-2} \rho_1 \s_2^j \\
	&= \s_1 \s_2 \s_1^{-1} \rho_1 \s_2^j \\
	&= \s_1 \s_1^2 \rho_1 \s_2^{-j} \rho_1 \s_2^j \\
	&= \s_1^3 \s_2^{2j}.
	\end{align*}
	It follows that $b \equiv 2j$ (mod $q$). Since $j = 1$ or $j = 1 + q/2$, we must have that $b \equiv 2$ (mod $q$), and since
	we can take $0 \leq b \leq q-1$, it follows that $b = 2$.
	
	If $p \neq 4$, we nevertheless have by \lref{s14-s26-normal} that $\langle \s_1^4 \rangle$ is a normal subgroup.
	Taking the quotient by this subgroup forces $p = 4$ without changing $b$, and so since $b = 2$ in the quotient, it follows
	that $b = 2$ in the original group.
	\end{proof}

	\begin{theorem}
	\label{thm:parameters}
	Let $\calP$ be a tight non-orientably regular polyhedron of type $\{p, q\}$ such that $\G(\calP) = \D(p,q)_{(i,j,a,b)}$.
	Then
	\begin{enumerate}
	\item $p$ is an odd multiple of $4$.
	\item $q$ is a multiple of $3$. Furthermore, if $p \neq 4$, then $q$ is an odd multiple of $6$.
	\item If $p/4 \equiv 3$ (mod $4$), then $i = p/4-1$, and otherwise $i = 3p/4-1$.
	\item If $p \neq 4$, then $j = 1 + q/2$, and if $p = 4$ then either $j = 1$ or $j = 1 + q/2$.
	\item $a = 1 + p/2$.
	\item $b = 2$.
	\end{enumerate}
	\end{theorem}

	\begin{proof}
	Parts (c) through (f) were proved in Lemmas \ref{lem:a-and-j-restrictions} to \ref{lem:value-of-b}. It remains to prove parts (a) and (b).

	\cref{p-and-q-divisors} tells us that $p$ is a multiple of $4$ and that $q$ is a multiple of $3$. Further, note that if $q$ is odd, then $j = 1$ (since $j = 1 + q/2$ is impossible in this case). Then \lref{a-and-j-restrictions} tells us that $p = 4$.
	So if $p \neq 4$, it must be that $q$ is a multiple of $6$.
	
	Now, suppose that $8$ divides $p$. Then since $\langle \s_1^4 \rangle$ is normal, so is $\langle \s_1^8 \rangle$. Taking the quotient
	by this normal subgroup then yields a tight non-orientably regular polyhedron of type $\{8, q\}$. In this quotient, $\langle \s_2^6 \rangle$
	is normal, and the quotient by this group yields a tight non-orientably regular polyhedron of type $\{8, 6\}$ or $\{8, 3\}$. But no
	such polyhedron exists (which we confirm by checking the Atlas of Small Regular Polytopes \cite{atlas}). Therefore, $8$ cannot divide
	$p$, and so $p$ is an odd multiple of $4$.
	
	We have already established that if $p \neq 4$, then $q$ cannot be odd, and so it must be a multiple of 6. Suppose that
	$q$ is a multiple of 12. Then since $\langle \s_2^6 \rangle$ is normal (by \lref{s14-s26-normal}), so is
	$\langle \s_2^{12} \rangle$, and the quotient by this normal subgroup yields a tight non-orientably regular polyhedron
	of type $\{p, 12\}$. Since $p \neq 4$, \lref{a-and-j-restrictions} says that $j = 7$. Now, $\langle \s_2^6 \rangle$ is a normal
	subgroup of our quotient by $\langle \s_2^{12} \rangle$, and in passing to the quotient by $\langle \s_2^6 \rangle$,
	we may replace $j$ with $1$. In that case, \lref{a-and-j-restrictions} says that $p = 4$ after all.
	Since $p \neq 4$, it follows that $q$ is not a multiple of 12.
	\end{proof}

	Thus, with the exception of the case where $p = 4$, there is only a single choice of parameters that (might) work,
	and in the case $p = 4$, there are 2 choices. It remains to show that there really are tight non-orientably regular polyhedra
	of these types $\{p, q\}$.
	
	\begin{lemma}
	\label{lem:non-or-existence}
	Let $r$ and $k$ be odd, let $p = 4r$ and let $q = 6k$. Let $i$, $j$, $a$, and $b$ satisfy the conditions of
	\tref{parameters}. Then $\D(p,q)_{(i,j,a,b)}$ is a string C-group of type $\{p, q\}$.
	\end{lemma}
	
	\begin{proof}
	Since $p$ is a multiple of $4$ and $q$ is a multiple of $6$, the group $\D(p,q)_{(i,j,a,b)}$ covers $\D(4,6)_{(i,j,a,b)}$.
	In the latter, we may reduce $i$ and $a$ modulo 4, and we may reduce $j$ and $b$ modulo 6.
	The parameter $i$ was chosen (in \lref{value-of-i}) such that $i \equiv 2$ (mod $4$), and since $a = 1 + p/2 = 1 + 2r$
	for some odd integer $r$, it follows that $a \equiv 3$ (mod $4$). The parameter $j$ satisfies
	$j \equiv 1$ (mod $q/2$), and thus $j \equiv 1$ (mod $3k$), which implies that $j \equiv 1$ (mod $3$).
	Therefore, $j \equiv 1$ or $4$ (mod $6$). Finally, $b = 2$. It follows that $\D(4,6)_{(i,j,a,b)} = \D(4,6)_{(2,1,3,2)}$
	or $\D(4,6)_{(2,4,3,2)}$. Using GAP \cite{gap}, we can verify that these latter two groups are the automorphism groups of (non-isomorphic) tight polyhedra of type
	$\{4, 6\}$, ($\{4, 6\}*48b$ and $\{4,6\}*48c$ in \cite{atlas}),
	so in $\D(p,q)_{(i,j,a,b)}$, the order of $\s_1$ is divisible by $4$ and the order of $\s_2$ is divisible by $6$.

	Now, let $G = \langle x, y \mid x^2 = y^2 = (xy)^r = 1 \rangle$. Then a small calculation shows that the function
	$\varphi: \D(p,q)_{(i,j,a,b)} \to G$ that sends $\rho_0$ to $x$, $\rho_1$ to $y$, and $\rho_2$ to $1$ is
	a surjective group homomorphism. From this it follows that the order of $\s_1$ is divisible by $r$. Since the
	order of $\s_1$ is also divisible by $4$ and is a divisor of $4r$, the order must be exactly $4r$.

	Similarly, let $H = \langle y, z \mid y^2 = z^2 = (yz)^{3k} = 1 \rangle$. Then the function
	$\varphi: \D(p,q)_{(i,j,a,b)} \to H$ sending $\rho_0$ to $1$, $\rho_1$ to $y$, and $\rho_2$ to $z$ is a surjective
	group homomorphism. Thus, the order of $\s_2$ is divisible by $3k$, and since it is also divisible by $6$
	and a divisor of $6k$, the order must be $6k$.
	
	We have established that $\D(p,q)_{(i,j,a,b)}$ has type $\{p, q\}$. To see that it is a string C-group, we note that
	$\D(p,q)_{(i,j,a,b)}$ covers $\D(4,q)_{(i,j,a,b)}$, which in turn covers $\D(4,6)_{(i,j,a,b)}$. Since $\D(4,6)_{(i,j,a,b)}$
	is a string C-group, two applications of \pref{quo-crit} shows that so is $\D(p,q)_{(i,j,a,b)}$.
	\end{proof}

	It remains to show that there is a tight non-orientably regular polyhedron of type $\{4, 3k\}$ whenever $3k$
	is odd. This follows directly from \cite[Thm. 5.1]{tight2}. Combined with \tref{parameters} and \lref{non-or-existence},
	we obtain the following result:
										
	 \begin{theorem}
	 There is a tight non-orientably regular polyhedron of type $\{p, q\}$ if and only if
	 \begin{enumerate}
	 \item $p = 4$ and $q = 3k$, or
	 \item $p = 4r$ and $q = 6k$, with $r > 1$ odd and $k$ odd, or
	 \item $q = 4$ and $p = 3k$, or
	 \item $q = 4r$ and $p = 6k$, with $r > 1$ odd and $k$ odd.
	 \end{enumerate}
	 Furthermore, in each case there is a unique such polyhedron up to isomorphism,
	 except in the cases where $p = 4$ and $q = 6k$, or $q = 4$ and $p = 6k$,
	 in which case there are two isomorphism types.
	 \end{theorem}

Now as in the case of tight orientably regular polyhedra, we consider the core of $\langle \sigma_i \rangle$ in the automorphism groups of tight non-orientably regular polyhedra.
Whenever $\langle \s_1 \rangle$ and $\langle \s_2 \rangle$ are not core-free in $\Delta(p,q)_{i,j,a,b}$, we can take the quotients of $\Delta$ by the two cores to obtain two tight non-orientably regular polyhedra.
As in the orientable case, we could reconstruct $\Delta(p,q)_{i,j,a,b}$ from these two quotients. The difficulty is that, unlike in the orientable case, some polyhedra with $\langle \sigma_2 \rangle$ core-free might have multiple edges.

Let us show how this can happen. \cref{fix-all-verts} showed that, if $\calP$ is orientably regular and if some 
$\varphi \in \langle \s_2 \rangle$ fixes one neighbor of $v$, then $\varphi$ fixes every vertex. However, this
need not be true in the non-orientable case. In particular, consider a polyhedron with double edges and with
$\varphi = \s_2^{q/2}$. Then it can happen that $\varphi$ acts as a rotation at $v$ but as a reflection through
one of its neighbors $u$. This is illustrated in Figure \ref{fig:RotRef}, where the two flags labeled $\Psi$ are identified, and $\varphi$ maps flag $\Phi$ into flag $\Psi$ by a half-turn around $v$, but also by a reflection by a vertical line through $u$. This gives us a polyhedron with double edges, even though $\langle \s_2 \rangle$ is core-free.

\begin{figure}
\begin{center}
\includegraphics[width=7cm, height=2.5cm]{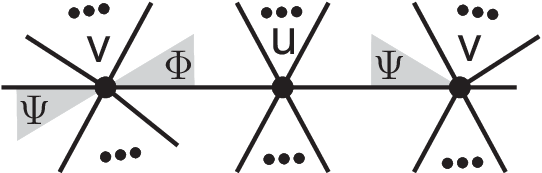}
\end{center}
\caption{Polyhedron with double edges but $\langle \sigma_2 \rangle$ core-free \label{fig:RotRef}}
\end{figure}

Examples of this situation are the duals of the polyhedra with automorphism groups $\Delta(4,q)_{2,1,3,2}$ for any $q$ divisible by $3$. The representation as maps of these polyhedra for $q\in \{3,6,9\}$ are shown in Figure \ref{fig:Grek} (see for example \cite{grek} for earlier appearance of these polyhedra as maps on surfaces). It is easy to see that in the polyhedron with group $\Delta(4,q)_{2,1,3,2}$, each square face shares opposite edges with another square, as shown in the upper part of Figure \ref{fig:Grek}. This implies that the dual has double edges. On the other hand, it is easy to verify that $\langle \sigma_1 \rangle$ is core-free in $\Delta(4,q)_{2,1,3,2}$.

\begin{figure}
\begin{center}
\includegraphics[width=7cm, height=4cm]{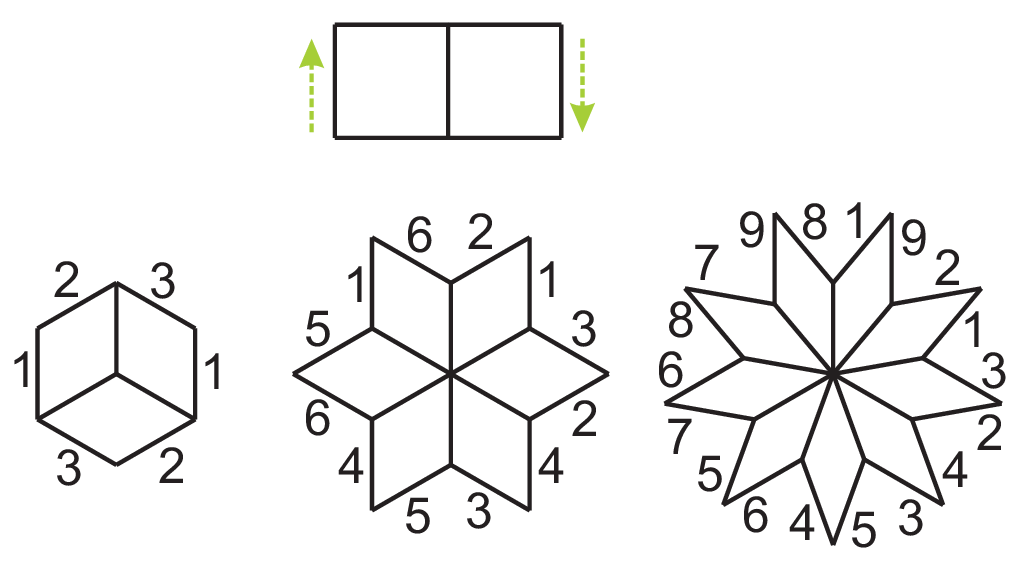}
\end{center}
\caption{Polyhedra with automorphism groups $\Delta(4,q)_{2,1,3,2}$ for $q \in \{3,6,9\}$ \label{fig:Grek}}
\end{figure}

\bibliographystyle{amsplain}

\providecommand{\bysame}{\leavevmode\hbox to3em{\hrulefill}\thinspace}
\providecommand{\MR}{\relax\ifhmode\unskip\space\fi MR }
\providecommand{\MRhref}[2]{%
  \href{http://www.ams.org/mathscinet-getitem?mr=#1}{#2}
}
\providecommand{\href}[2]{#2}

\ \\
Gabe Cunningham \\
{\em  University of Massachusetts Boston \\
Boston, Massachusetts. \\
e-mail: }{\tt gabriel.cunningham@gmail.com}

\ \\
Daniel Pellicer \\
{\em  Centro de Ciencias Matem\'aticas, \\
Universidad Nacional Aut\'onoma de M\'exico, \\
Morelia, Mexico. \\
e-mail: }{\tt pellicer@matmor.unam.mx}

\end{document}